\documentclass[a4paper,11pt]{article}

\usepackage[sort,compress]{cite}
 
\usepackage{amsmath,amsfonts,amssymb,amsthm}
\usepackage{url}
\usepackage{color}
\usepackage{enumerate}
\usepackage{blkarray}
\usepackage{hyperref,pgf}
\usepackage{setspace,xspace}

\DeclareMathOperator{\im}{im} 
\DeclareMathOperator{\sign}{sign}
\DeclareMathOperator{\e}{e} 
\DeclareMathOperator{\diag}{diag}
\DeclareMathOperator{\rank}{rank}

\newcommand{\RR}{{\mathbb R}}
\newcommand{\NN}{{\mathbb N}}

\newcommand{\mal}{\circ}
\newcommand{\st}{\mid}
\newcommand{\dd}[2]{\frac{\text{d} #1}{\text{d} #2}}
\newcommand{\DD}[2]{\frac{\partial #1}{\partial #2}}
\newcommand{\ve}{\varepsilon}

\theoremstyle{plain}
\newtheorem{theorem}{Theorem}[section]
\newtheorem{proposition}[theorem]{Proposition}
\newtheorem{corollary}[theorem]{Corollary}
\newtheorem{lemma}[theorem]{Lemma}
\newtheorem*{descartes}{Descartes' rule of signs}

\theoremstyle{definition}
\newtheorem{definition}[theorem]{Definition}
\newtheorem{remark}[theorem]{Remark}
\newtheorem{example}[theorem]{Example}

\def\refi{(i)\@\xspace}
\def\refii{(ii)\@\xspace}
\def\refiii{(iii)\@\xspace}
\def\refsig{(sig)\@\xspace}
\def\refinj{(inj)\@\xspace}
\def\refdet{(det)\@\xspace}
\def\refmin{(min)\@\xspace}
\def\reflin{(lin)\@\xspace}
\def\refjac{(jac)\@\xspace}
\def\refbound{(bnd)\@\xspace}
\def\refsurj{(ex)\@\xspace}

\makeatletter
\def\blfootnote{\xdef\@thefnmark{}\@footnotetext}
\makeatother

\begin{document}

\title{Sign conditions for injectivity of generalized polynomial maps \\
with applications to chemical reaction networks \\ and real algebraic geometry}

\renewcommand{\thefootnote}{\fnsymbol{footnote}}
\author{%
Stefan M\"uller\footnotemark[1], Elisenda Feliu\footnotemark[1], 
Georg Regensburger\footnotemark[1],\\
Carsten Conradi,
Anne Shiu, Alicia Dickenstein\footnotemark[2]
}

\blfootnote{
\scriptsize

\smallskip
\noindent
{\bf S. M\"uller} (\href{mailto:stefan.mueller@ricam.oeaw.ac.at}{stefan.mueller@ricam.oeaw.ac.at}).  
Johann Radon Institute for Computational and Applied Mathematics, Austrian Academy of Sciences, Altenbergerstra{\ss}e 69, 4040 Linz, Austria

\smallskip
\noindent
{\bf E. Feliu} (\href{efeliu@math.ku.dk}{efeliu@math.ku.dk}). 
Dept.\ of Mathematical Sciences, University of Copenhagen, Universitetsparken 5, 2100 Denmark 

\smallskip
\noindent
{\bf G. Regensburger} (\href{mailto:georg.regensburger@ricam.oeaw.ac.at}{georg.regensburger@ricam.oeaw.ac.at}).
Johann Radon Institute for Computational and Applied Mathematics, Austrian Academy of Sciences, Altenbergerstra{\ss}e 69, 4040 Linz, Austria

\smallskip
\noindent
{\bf C. Conradi} (\href{mailto:conradi@mpi-magdeburg.mpg.de}{conradi@mpi-magdeburg.mpg.de}).
Max-Planck-Institut Dynamik komplexer technischer Systeme, Sandtorstr.\ 1, 39106 Magdeburg, Germany 

\smallskip
\noindent
{\bf A. Shiu} (\href{mailto:annejls@math.tamu.edu}{annejls@math.tamu.edu}). 
Dept.\ of Mathematics, Mailstop 3368, Texas A\&M Univ., College Station, TX~77843--3368, USA

\smallskip
\noindent
{\bf A. Dickenstein} (\href{mailto:alidick@dm.uba.ar}{alidick@dm.uba.ar}).
Dto.\ de Matem\'atica, FCEN, Universidad de Buenos Aires, Ciudad Universitaria, Pab.\ I, C1428EGA Buenos Aires, Argentina, and IMAS/CONICET 

}

\footnotetext[1]{These authors contributed equally to this work.}
\footnotetext[2]{Corresponding author}

\date{\today}

\maketitle

\begin{abstract}
We give necessary and sufficient conditions in terms of sign vectors
for the injectivity of families of polynomial maps with arbitrary real exponents defined on the positive orthant.
Our work relates and extends existing injectivity conditions
expressed in terms of Jacobian matrices and determinants.
In the context of chemical reaction networks with power-law kinetics,
our results can be used to preclude as well as to guarantee  multiple positive steady states.
In the context of real algebraic geometry,
our work recognizes a prior result of Craciun, Garcia-Puente, and Sottile, 
together with work of two of the authors, as the first partial multivariate generalization of the classical Descartes' rule,
which bounds the number of positive real roots of a univariate real polynomial
in terms of the number of sign variations of its coefficients.

\end{abstract}

\section{Introduction}
In many fields of science,
the analysis of parametrized systems by way of sign vectors has a long history. 
In economics,
market models depend on monotonic price and demand curves,
leading to the theory of sign-solvable linear systems \cite{sign-023,sign-009}.
In electronics,
devices such as diodes, transistors, and operational amplifiers are characterized by monotonic functions,
and one studies whether the input-output relation of an electronic circuit is well posed,
using the theory of oriented matroids~\cite{lin-018,Chaiken1996}.
In many settings, uniqueness of positive solutions is a desirable property, but deciding this is difficult in general~\cite{CGS, tackling}.  
If, however, the maps of interest are injective, then this precludes multiple solutions. 

Motivated by applications to chemical reaction networks and real algebraic geometry,
we characterize injectivity of parametrized families of polynomial maps with arbitrary real exponents, 
in terms of sign vectors.
Our work builds on results from chemical engineering,
by abstracting, relating, and extending existing injectivity conditions expressed in terms of Jacobian matrices and determinants.

The relevant literature from the theory of chemical reaction networks is discussed in Subsection~\ref{subsec:crnt}.
The main application  
to real algebraic geometry is addressed in Subsection~\ref{subsec:rag}.

\subsection{Statement of the main theorem}

Throughout this paper we consider families of maps defined on the positive orthant,
associated with two real matrices of coefficients and exponents, respectively, and a vector of positive parameters.
\begin{definition} \label{def:main_gen}
Let $A = (a_{ij}) \in \RR^{m\times r}$, $B= (b_{ij}) \in\RR^{r\times n}$, and $\kappa \in \RR_+^r$.
We define the associated {\em generalized polynomial map} $f_\kappa \colon \RR_+^n \to \RR^m$ as
\begin{equation*}
f_{\kappa,i}(x) = \sum_{j =1}^r a_{ij} \, \kappa_j \, x_1^{b_{j1}} \cdots x_n^{b_{jn}}, \quad i = 1, \dots, m .      
\end{equation*}
\end{definition}
The term {\em generalized} indicates that we allow polynomials with real exponents.
In the literature, generalized polynomials occur under other names.
For instance, they are called \emph{signomials} in geometric programming~\cite{Boydetal2007}.

We often use a more compact notation.
By introducing $A_\kappa \in \RR^{m \times r}$ as $A_\kappa = A \diag(\kappa)$
and $x^B \in \RR^r_+$ via $(x^B)_j = x_1^{b_{j1}} \cdots x_n^{b_{jn}}$ for $j=1,\ldots,r$,
we can write
\begin{equation} \label{eq:map_f}
f_\kappa(x) = A_\kappa \, x^B .
\end{equation}

A generalized polynomial map $f_\kappa \colon \RR_+^n \to \RR^n$ \eqref{eq:map_f} with $A \in \RR^{n \times r}$ and $B \in \RR^{r \times n}$,
induces a system of ordinary differential equations (ODEs)
called a {\em power-law system}:
\begin{equation} \label{eq:diffeq}
\dd{x}{t} = f_\kappa(x).
\end{equation}
For any initial value $x_0\in \RR^n_+$, 
the solution is confined to the coset $x_0+S_\kappa$,
where $S_\kappa$ is the smallest vector subspace containing the image of $f_\kappa$.
Hence, when studying positive steady states of (\ref{eq:diffeq}),
one is in general interested in the positive solutions to the equation $f_\kappa(x) = 0$ within cosets $x^\prime + S_\kappa$ with $x^\prime \in \RR^n_+$. 
Due to the form of $f_\kappa$, one has $S_\kappa \subseteq S$ where $S = \im(A)$.
In many applications, $S_\kappa = S$ for all $\kappa \in \RR_+^r$, for example, if the rows of $B$ are distinct.
If $f_\kappa$ is injective on  $(x^\prime + S)\cap \RR^n_+$, then $f_\kappa(x) \neq f_\kappa(y)$ for all distinct $x$, $y\in (x^\prime+S)\cap \RR^n_+$,
and hence the coset $x^\prime +S$ contains at most one positive steady state.
Clearly, for a vector subspace $S$ of $\RR^n$, two vectors $x$, $y \in \RR^n$ lie in $x^\prime +S$ for some $x'\in \RR^n$, if and only if $x-y\in S$.
This motivates the following definition of injectivity with respect to a subset.
 
\begin{definition} \label{def:main_inj}
Given two subsets $\Omega, S \subseteq \RR^n$, a function $g$ defined on $\Omega$ is called \emph{injective with respect to} $S$
if $x,y \in \Omega$, $x \neq y$, and $x-y \in S$ imply $g(x) \neq g(y)$.
\end{definition}

We will in general consider functions defined on the positive orthant, that is, $\Omega = \RR_+^n$.
When $S$ is a vector subspace, injectivity with respect to $S$ is equivalent to injectivity on every coset $x^\prime +S$.

When the matrix $B$ has integer entries, determining the injectivity of the map $f_{\kappa}$ 
for a fixed (computable) parameter value $\kappa$, with respect to a semialgebraic subset $S$, 
is a question of quantifier elimination and thus can be decided algorithmically, but is very hard in practice. 
This paper focuses on how to decide injectivity for the whole  family, that is, for \emph{all} possible values of $\kappa \in \RR_+^r$,
for a matrix $B$ with real entries.
Our results are given in terms of sign vectors characterizing the orthants
that $\ker(A)$ and (a subset of) $\im(B)$ intersect nontrivially.
 
\begin{definition} \label{def:main_sig} 
For a vector $x \in \RR^n$, we obtain the \emph{sign vector} $\sigma(x)\in \{-,0,+\}^n$ by applying the sign function componentwise. 
\end{definition}
Note that a sign vector $\nu\in \{-,0,+\}^n$ corresponds to the (possibly lower-dimensional) orthant of $\RR^n$ given by $\sigma^{-1}(\nu)$.
For a subset $S \subseteq \RR^n$, we write $\sigma(S) = \{ \sigma(x) \st x \in S \}$ for the set of all sign vectors of $S$ and 
\begin{equation*} \label{eq:Sigma}
\Sigma(S) = \sigma^{-1}(\sigma(S))
\end{equation*} 
for the union of all (possibly lower-dimensional) orthants that $S$ intersects. 
For convenience, we introduce $S^* = S \setminus \{0\}$.

In order to state our main result, we require some more notation.
Identifying $B \in \RR^{r \times n}$ with the linear map $B \colon \RR^n \to \RR^r$,
we write $B(S)$ for the image under $B$ of the subset $S \subseteq \RR^n$.
In analogy to $A_\kappa$, we introduce $B_\lambda = B \diag(\lambda)$ for  
$\lambda \in \RR^n_+$.
Finally, we write $J_{f_\kappa}$ for the Jacobian matrix associated with the map $f_\kappa$.
Here is our main result, which brings together and extends various existing results (see Section~\ref{subsec:crnt}).

\begin{theorem} \label{thm:main}
Let $f_\kappa \colon \RR^n_+ \to \RR^m$ be the generalized polynomial map $f_\kappa(x) = A_\kappa \, x^B$,
where $A \in \RR^{m \times r}$, $B \in \RR^{r \times n}$, and $\kappa \in \RR^r_+$.
Further, let $S \subseteq \RR^n$.
The following statements are equivalent:
\begin{itemize}
\item[(inj)]
$f_\kappa$ is injective with respect to $S$, for all $\kappa \in \RR^r_+$.
\item[(jac)]
$\ker\left(J_{f_\kappa}(x) \right) \cap S^* = \emptyset$, for all $\kappa \in \RR^r_+$ and $x \in \RR^n_+$.
\item[(lin)]
$\ker(A_\kappa B_\lambda) \cap S^* = \emptyset $, for all $\kappa \in \RR^r_+$ and $\lambda\in \RR^n_+$.
\item[(sig)] $\sigma(\ker(A)) \cap \sigma(B(\Sigma(S^*))) = \emptyset$. 
\end{itemize}
\end{theorem}

Note that, for a fixed exponent matrix $B$, condition \refsig depends only on 
the sign vectors of $\ker(A)$ and $S$. In particular, $f_\kappa$ is injective with respect to $S$ for all $\kappa \in \RR^r_+$
if and only if it is injective with respect to $\Sigma(S) \subseteq \RR^n$,
which is the largest set having the same sign vectors as $S$.

To study unrestricted injectivity, we set $S = \RR^n$ in Theorem~\ref{thm:main},
in which case condition \refsig is equivalent to
\begin{equation*}
\ker(B) = \{0\} \quad \textrm{and} \quad \sigma(\ker(A)) \cap \sigma(\im(B)) = \{0\} ;
\end{equation*}
see Corollary~\ref{cor:unrestricted}.
Assuming $\ker(B)=\{0\}$,
condition \refsig depends only on the corresponding vector subspaces $\ker(A)$ and $\im(B)$;
see also~\cite[Theorem~3.6]{MR}.

Birch's theorem~\cite{Birch} in statistics corresponds to the unrestricted case
$S=\RR^n$ and $B=A^T$ with full rank $n$.
Note that $\im(B)= \im(A^T)=\ker(A)^\bot$,
and hence $\sigma(\ker(A)) \cap \sigma(\im(B)) = \{0\}$ is trivially fulfilled.
Therefore, statement \refinj holds, so Theorem~\ref{thm:main} guarantees that
for any choice of vectors $y \in \RR^n$ and $\kappa \in \RR_+^r$,
there is at most one solution $x \in \RR_+^n$ to the equations
\begin{equation*}
\sum_{j =1}^r a_{ij} \, \kappa_j \,  x_1^{a_{1j}} \cdots x_n^{a_{nj}} = y_i, \quad i = 1, \ldots, n . 
\end{equation*}
In fact, Birch's theorem also guarantees the existence of a solution, for all $y$ 
in the interior of the polyhedral cone generated by the columns of $A$.
A related result, due to Horn, Jackson, and Feinberg, 
asserts the existence and uniqueness of complex balancing equilibria~\cite{Feinberg1972,Horn1972,HornJackson1972}, 
which is discussed in the next subsection and Section~\ref{sec:app}. 
Our generalization of Birch's theorem based on~\cite{MR} is given in statement \refsurj of Theorem~\ref{thm:Descartes}.

\subsection{Motivation from chemical reaction networks} \label{subsec:crnt}

For chemical reaction networks with mass-action kinetics,
the concentration dynamics are governed by dynamical systems \eqref{eq:diffeq} with polynomial maps $f_\kappa(x)=A_\kappa \, x^B$,
as defined in \eqref{eq:map_f}. 
We introduce some terms that are standard in the chemical engineering literature.
The components of $\kappa\in \RR^r_+$  are called   {\em rate constants} and are often unknown in practice.
The vector subspace $S=\im(A)$ is called the \emph{stoichiometric subspace}.
One speaks of {\em multistationarity} if there exist a vector of rate constants $\kappa\in \RR^r_+$
and two distinct positive vectors $x,y\in \RR^n_+ $ with $x-y \in S$ such that $f_\kappa(x)=f_\kappa(y)=0$.
Clearly, if  $f_\kappa$ is injective with respect to $S$ for all values of $\kappa$, then multistationarity is ruled out.
Therefore, Theorem~\ref{thm:main} can be applied in this setting to preclude multistationarity.

Indeed, our work unifies and extends existing conditions for injectivity established in the context of chemical reaction networks. 
The first such result was given by Craciun and Feinberg for the special case of a fully open network,
that is, when each chemical species has an associated outflow reaction and hence $S=\RR^n$: 
injectivity of the corresponding family of polynomial maps
was characterized by the  nonsingularity of the  associated Jacobian matrices, which could be assessed by determinantal conditions~\cite{ME_I}.
An elementary proof of this foundational result appeared in the context of geometric modeling~\cite{CGS},
and extended Jacobian and determinantal criteria were subsequently 
achieved for arbitrary networks~\cite{gnacadja_linalg,FeliuWiuf_MAK,JoshiShiu}.
Also, for networks with uni- and bi-molecular reactions and fixed rate constants,
injectivity of the polynomial map has been characterized~\cite{PKC}.
Injectivity results have been obtained also for families of kinetics different from mass-action,
in particular,
for nonautocatalytic kinetics~\cite{BanajiDonnell,BanajiCraciun2010},
power-law kinetics and strictly monotonic kinetics~\cite{WiufFeliu_powerlaw,feliu-bioinfo},
weakly monotonic kinetics~\cite{ShinarFeinberg2012}, and other families~\cite{banaji_pantea}.
Further, several injectivity criteria have been translated to
conditions on the species-reaction graph or the interaction 
graph~\cite{ME_II,MinchevaCraciun2008,BanajiCraciun2010,HeltonDeterminant,ShinarFeinberg2013}. 

Sign conditions for the injectivity of monomial maps have been applied both to preclude and to assert 
multiple positive steady states for several special types of  steady states, such as \emph{detailed balancing} 
and \emph{complex balancing equilibria} of mass-action systems \cite{Feinberg1972,Horn1972,HornJackson1972}, 
\emph{toric steady states} of mass-action systems~\cite{TSS}, and complex balancing equilibria of generalized 
mass-action systems~\cite{MR}. Specifically, such special steady states are parametrized by a monomial map, and 
multistationarity occurs if and only if the sign vectors of two vector subspaces intersect nontrivially.
Moreover, for given rate constants, 
existence of one complex balancing equilibrium in a mass-action system 
implies existence and uniqueness of such steady states within each coset of the stoichiometric subspace,
and no other steady states are possible \cite{HornJackson1972}.

In this paper we unify and extend the criteria for injectivity and multistationarity described above. 
Related results appear in the deficiency-oriented theory,
as initiated by Horn, Jackson, and Feinberg~\cite{Feinberg1972,HornJackson1972,Horn1972}
(see also~\cite{feinbergnotes,Feinberg1987,Feinberg1988,Feinberg1995a,Feinberg1995b}).
This theory is named after the {\em deficiency} of a reaction network, a nonnegative
integer that can be computed from basic network properties.
Deficiency zero networks  with mass-action
kinetics admit positive steady states if and only if the network is strongly connected, and, in this case,
there is a unique positive steady state, which is a complex balancing equilibrium.
On the other hand, 
some networks with deficiency one admit multiple positive  steady states,
and the capacity for multistationarity is characterized by certain sign conditions~\cite{Feinberg1988,Feinberg1995b}. 
For other uses of sign conditions to determine multistationarity,
see~\cite{MAPK,fein-043,fein-050} and the related applications to particular biochemical networks~\cite{HFC}.

\subsection{Application to real algebraic geometry} \label{subsec:rag}

An interesting consequence in the realm of real algebraic geometry that
emerges from the study of injectivity of generalized polynomial maps in applications is Theorem~\ref{thm:Descartes} below.
Statement~\refbound in that result was first proved by Craciun, Garcia-Puente, and Sottile in their study of
control points for toric patches \cite[Corollary 8]{CGS} based on a previous injectivity result by Craciun and 
Feinberg~\cite{ME_I}.
The surjectivity result underlying statement~\refsurj is due to M\"uller and Regensburger~\cite[Theorem~3.8]{MR},
who use arguments of degree theory for differentiable maps. 
We recognize Theorem~\ref{thm:Descartes} as the first
partial multivariate generalization of the following well-known rule proposed by Ren\'e Descartes in 1637
in ``La G\'eometrie'', an appendix to his ``Discours de la M\'ethode'', see \cite[pp.\ 96--99]{struik}.
No multivariate generalization is known,
and only a lower bound together with a disproven conjecture was proposed by Itenberg and Roy in 1996~\cite{ir96}. 

\begin{descartes}\label{pro:Descartes_original}
Given a univariate
real polynomial $f(x) = c_0 + c_1 x + \cdots + c_r x^r$,
the number of positive real roots of $f$ (counted with multiplicity) is bounded above by the number of sign variations
in the ordered sequence of the coefficients $c_0,\dots, c_r$, more precisely, discard the zeros in this sequence and then count the number of times two consecutive entries have different signs. 
Additionally, the difference between these two numbers (the number of positive roots and the number of sign variations) is even.
\end{descartes}
For instance, given the polynomial $f(x) = c_0 + x - x^2 + x^{k}$ with degree $k>2$, the number of variations in the sequence ${\rm sign}(c_0),+,-,+$ equals $3$ if $c_0 < 0$ and $2$ if $c_0 \ge 0$. Hence, $f$~admits at most $3$ or $2$ positive real roots, respectively, and this is independent of its degree.

An important consequence of Descartes' rule of signs is that the number of real roots of a real univariate polynomial $f$ can be bounded in terms of the number of monomials
in $f$ (with nonzero coefficient), independently of the degree of $f$. In the multivariate case, Khovanskii~\cite[Corollary 7]{Fewnomials} proved the remarkable result that the number of nondegenerate solutions in $\RR^n$ of a system of $n$ real polynomial equations  can also be bounded  
solely in terms of the number $q$ of distinct monomials appearing in these equations. Explicitly, the number of nondegenerate positive roots is 
at most  $2^{(q-1)(q-2)/2} \, (n+1)^{q-1}$.
In contrast to Descartes' rule, this bound is far from sharp and
the only known refinements of this bound do not depend on the signs of the coefficients of $f$ \cite[Chapters 5--6]{SottileBook}.

Accordingly, we view Theorem~\ref{thm:Descartes} as the first partial multivariate generalization of Descartes' rule, as the conditions of the theorem for precluding more than one positive solution depend both on the coefficients and the monomials of $f$.

We require the following notation.  
We introduce $[r] = \{1,\ldots,r\}$ for any natural number~$r$. For $A \in \RR^{n \times r}$ and $B \in \RR^{r \times n}$ with $n \le r$, and 
some index set $J\subseteq [r]$ of cardinality $n$, we write
$A_{[n],J}$ for the submatrix of $A$ indexed by the columns in $J$ and $B_{J,[n]}$ for the submatrix of $B$ indexed by the rows in $J$.

For any choice of $y \in \RR^n$, we consider
the system of $n$ equations in $n$ unknowns 
\begin{equation}\label{eq:y}
\sum_{j =1}^r a_{ij} \, x_1^{b_{j1}} \cdots x_n^{b_{jn}} = y_i, \quad i = 1, \dots,n.
\end{equation}
We denote by $C^\circ(A)$ the interior of the polyhedral cone
generated by the column vectors $a^1, \dots, a^r$ of $A$:
\begin{equation*}\label{eq:cone}
C^\circ(A) = \left\{ \sum_{i=1}^r \mu_i \, a^i \in \RR^n \st \mu \in \RR^r_+ \right\}.
\end{equation*}

\begin{theorem} \label{thm:Descartes}[Multivariate Descartes' rule for (at most) one positive real root]
Let $A \in \RR^{n \times r}$ and $B \in \RR^{r \times n}$ be matrices with full rank $n$. Then,
\begin{itemize}
\item[\refbound] 
Assume that for all index sets $J \subseteq [r]$ of cardinality $n$,
the product $\det\big(A_{[n],J}\big) \det\big(B_{J,[n]}\big)$ either is zero or has the same sign as all other nonzero products,
and moreover, at least one such product is nonzero. Then,~\eqref{eq:y}
has at most one positive solution $x \in \RR_+^n$, for any $y \in \RR^n$.
\item[\refsurj]
Assume that the row vectors of $B$ lie in an open half-space and that
the determinants $\det\big(A_{[n],J}\big)$ and $\det\big(B_{J,[n]}\big)$ have the same sign
for all index sets $J \subseteq [r]$ of cardinality $n$, or the opposite sign in all cases.
Then, \eqref{eq:y}
has exactly one positive solution $x \in \RR_+^n$ if and only if $y \in C^\circ(A)$.
\end{itemize}
\end{theorem}

Note that the sign conditions in statement~\refsurj together with the full rank of the matrices
imply the hypotheses of~\refbound.

To analyze a univariate polynomial $f(x)=c_0 +c_1 x + \cdots + c_r x^r$ 
in the setting of Theorem~\ref{thm:Descartes}, we have 
$A \in \RR^{1 \times r}$ with entries $c_1, \dots, c_r$,
$B\in \RR^{r \times 1}$ with entries $1, \dots, r$,
and $y = -c_0$.
In this univariate case, the hypotheses of \refbound in 
Theorem~\ref{thm:Descartes} reduce to the conditions
that $c_1, \dots, c_r$ are all
nonnegative (or nonpositive) and not all are zero. 
If these hold, Theorem~\ref{thm:Descartes} states that $f$ has at most one positive real root 
and, furthermore, if $c_0$ has the opposite sign from the nonzero $c_1, \dots, c_r$'s, \refsurj guarantees the existence of this root.
Indeed, there is at most one sign variation, depending on $\sign(c_0)$, and so the classical Descartes' rule yields the same conclusion.  
The result is also valid in the case of real, not necessarily natural, exponents.

In Proposition~\ref{pro:polyequ}, we consider the more general system of $m$ equations in $n$~unknowns with $r$ parameters:
$f_\kappa(x) = y$, where $f_\kappa$ is as in Definition~\ref{def:main_gen}.
More precisely, we give a criterion via sign vectors 
for precluding multiple positive real solutions 
$x \in \RR^n_+$
for all $y \in \RR^m$ and $\kappa \in \RR^r_+$.

We will give the proof of Theorem~\ref{thm:Descartes} in Section~\ref{subsec:solve}, where we restate the sign conditions on the minors
of $A$ and $B$ in terms of oriented matroids. Based on this approach, a generalization for
multivariate polynomials systems in $n$ variables with $n+2$ distinct
monomials is given in~\cite{BD14}. This case shows the intricacy inherent in the pursuit of 
a full generalization of Descartes' rule to the multivariate case.

\paragraph{Outline of the paper.}
In Section~\ref{sec:inj},
we characterize, in terms of sign vectors,
the injectivity of a family of generalized polynomial maps with respect to a subset.
In particular,
we prove Theorem~\ref{thm:main}, thereby isolating and generalizing key ideas in the literature.
Further, we relate our results to determinantal conditions, in case the subset is a vector subspace.  
In Section~\ref{sec:app},
we apply our results to chemical reaction networks with power-law kinetics,
thereby relating and extending previous results.
We give conditions for precluding multistationarity in general,  
for precluding multiple ``special'' steady states,  
and for guaranteeing the existence of two or more such steady states.
Further, we present applications to real algebraic geometry.
We prove the partial multivariate generalization of Descartes' rule, Theorem~\ref{thm:Descartes}, and we restate the hypotheses in the language of oriented matroids.
Finally, in Section~\ref{sec:algo}, we address algorithmic aspects of our results,
in particular, the efficient computation of sign conditions to decide injectivity.

\section{Sign conditions for injectivity} \label{sec:inj}

In this section,
we characterize, in terms of sign vectors, generalized polynomial maps $f_\kappa(x) = A_\kappa \, x^B$
that are injective with respect to a subset for all choices of the positive parameters $\kappa$.
We accomplish this through a series of results that lead to the proof of Theorem~\ref{thm:main}.

\subsection{Notation}\label{subsec:notation}

Here we summarize the notation used throughout this work.
Moreover, we elaborate on the concept of sign vectors defined in the introduction.

We denote the strictly positive real numbers by $\RR_+$
and the nonnegative real numbers by $\overline{\RR}_+$.
We define $\e^x \in \RR^n_+$ for $x \in \RR^n$ componentwise, that is, $(\e^x)_i = \e^{x_i}$;
analogously, $\ln(x) \in \RR^n$ for $x \in \RR^n_+$ and $x^{-1} \in \RR^n$ for $x \in \RR^n$ with $x_i \neq 0$.
For $x,y \in \RR^n$,
we denote the componentwise (or Hadamard) product by $x \mal y \in \RR^n$, that is, $(x \mal y)_i = x_i y_i$.
Further,
we define $x^b \in \RR$ for $x \in \RR^n_+$ and $b \in \RR^n$ as $x^b = \prod_{i=1}^n x_i^{b_i}$.

Given a matrix $B \in \RR^{r \times n}$, we denote by $b^1, \dots, b^n$ its column vectors
and by $b_1, \dots, b_r$ its row vectors.
Thus, the $j$th coordinate of the map $x^B \colon\RR^n_+ \to \RR^r_+$
is given by
\[
(x^B)_j  = x^{b_{j}} =x_1^{b_{j1}} \cdots x_n^{b_{jn}}.
\]
Recall that we define $B_\lambda$ for $B \in \RR^{r \times n}$ and $\lambda \in \RR_+^n$
as $B_\lambda = B \diag(\lambda)$.

We identify a matrix $B \in \RR^{r \times n}$ with the 
corresponding linear map $B \colon \RR^n \to \RR^r$
and write $\im(B)$ and $\ker(B)$ for the respective vector subspaces.
For a subset $S \subseteq \RR^n$,
we write $S^* = S \setminus \{0\}$
and denote the image of $S$ under $B$ by
\[
B(S) = \{ B \, x \st x \in S \}.
\]
For any natural number $n$, we define $[n]=\{1,\dots,n\}$. 
Given sets $I \subseteq [n]$ and $J\subseteq [r]$,
we denote the submatrix of $B$ with row indices in $J$ and column indices in $I$ by $B_{J,I}$.

Now we are ready to state some consequences of Definition \ref{def:main_sig}.
For $x,y \in \RR^n$, we have the equivalence
\[
\sigma(x) = \sigma(y)
\quad \Leftrightarrow \quad
x = \lambda \mal y \textrm{ for some } \lambda \in \RR^n_+ ,
\]
and hence, for $S \subseteq \RR^n$, we obtain
\begin{equation}\label{eq:SigmaS}
\Sigma(S) = \sigma^{-1}(\sigma(S)) = \{\lambda \mal x \st \lambda \in \RR^n_+ \text{ and } x \in S\}.
\end{equation}
For subsets $X,Y \subseteq \RR^n$, we have the equivalences
\begin{equation} \label{sigmaSigma}
\Sigma(X) \cap Y = \emptyset
\quad \Leftrightarrow \quad
\sigma(X) \cap \sigma(Y) = \emptyset
\quad \Leftrightarrow \quad
X \cap \Sigma(Y) = \emptyset .
\end{equation}

\subsection{Families of linear maps} \label{subsec:sign_linear}

In this subsection, we consider the case of linear maps.
We start with a useful lemma. 

\begin{lemma} \label{lem:lin} 
Let $B \in \RR^{r \times n}$ and $S \subseteq \RR^n$. 
The following statements are equivalent:
\begin{enumerate}[(i)]
\item $\ker(B_\lambda) \cap S = \emptyset$, for all $\lambda \in \RR^n_+$.
\item $\sigma(\ker(B)) \cap \sigma(S) = \emptyset$.
\end{enumerate}
\end{lemma}

\begin{proof}
Statement \refi holds if and only if $B_\lambda \, x = B(\lambda \mal x) \not = 0$
for all $\lambda \in \RR^n_+$ and $x \in S$,
that is,
if and only if
$\ker(B) \cap \Sigma(S) = \emptyset$.
By~\eqref{sigmaSigma},
this is equivalent to statement \refii.
\end{proof}

We note that, if $0 \in S$, statements \refi and \refii do not hold, 
so we instead apply Lemma~\ref{lem:lin} to $S^*$.
In particular, if $S$ is a vector subspace of $\RR^n$,
then $\ker(B_{\lambda}) \cap S^* = \emptyset$ reduces to $\ker(B_{\lambda}) \cap S = \{0\}$,
that is, $B_{\lambda}$ is injective on $S$.

Now we are ready to prove the equivalence of statements \reflin and \refsig in Theorem~\ref{thm:main}.

\begin{proposition} \label{pro:lin}
Let $A \in \RR^{m \times r}$, $B \in \RR^{r \times n}$, and $S \subseteq \RR^n$.
The following statements are equivalent:
\begin{enumerate}[(i)]
\item
$\ker(A_\kappa B_\lambda) \cap S = \emptyset$, for all $\kappa \in \RR^r_+$ and $\lambda\in \RR^n_+$. 
\item
$\sigma(\ker(A)) \cap \sigma(B(\Sigma(S))) = \emptyset$.
\end{enumerate}
\end{proposition}

\begin{proof}
Clearly, statement \refi is equivalent to
$\ker(A_\kappa) \cap  B_\lambda(S) = \emptyset$, for all $\kappa \in \RR^r_+$ and ${\lambda \in \RR^n_+}$.
Using~\eqref{eq:SigmaS}, this is equivalent to $\ker(A_\kappa) \cap B(\Sigma(S)) = \emptyset$, for all $\kappa \in \RR^r_+$.
By Lemma~\ref{lem:lin} applied to the matrix $A$ and the subset $B(\Sigma(S))$,
this is in turn equivalent to statement \refii.
\end{proof}

Again, if $S$ is a vector subspace, $\ker(A_\kappa B_\lambda) \cap S^* = \emptyset$
reduces to $\ker(A_\kappa B_\lambda) \cap S = \{0\}$,
that is, $A_\kappa B_\lambda$ is injective on $S$.
Clearly, the statements in Lemma~\ref{lem:lin} are necessary conditions for the statements in Proposition~\ref{pro:lin}.

\subsection{Families of generalized monomial/polynomial maps} \label{subsec:sign_monopoly}

In this subsection, we use the results on families of linear maps
to give sign conditions for the injectivity of families of generalized polynomial maps with respect to a subset.

From Definition~\ref{def:main_inj},
we conclude that a function $g$ defined on $\RR^n_+$ is injective with respect to a subset $S \subseteq \RR^n$
if and only if for every $x \in \RR^n_+$ one has $g(x) \neq g(y)$ for all $y \in (x+S^*) \cap \RR^n_+$,
where $x+S^* := \{x+y \st y \in S^*\}$.
In case $S$ is a vector subspace, then such a function 
$g$ is injective on the intersection $(x+S) \cap \RR^n_+$ of any coset $x+S$ with the domain $\RR^n_+$.

We start with a key observation.

\begin{lemma} \label{lem:lambda}
For $S \subseteq \RR^n$, let
\begin{equation} \label{eq:Lambda}
\Lambda(S) := \{ \ln x - \ln y \st x,y \in \RR^n_+ \textrm{ and } x - y \in S \} .
\end{equation} 
Then, $\Lambda(S) = \Sigma(S)$.
\end{lemma}

\begin{proof}
Let $x,y \in \RR^n_+$ such that $x-y \in S$. Then, using the strict monotonicity of the logarithm 
we have $\sigma(\ln x - \ln y) = \sigma(x-y) \in \sigma(S)$ and  hence $\ln x - \ln y \in \Sigma(S)$.
This proves the inclusion $\Lambda(S) \subseteq \Sigma(S)$.
Conversely,
let $\lambda \in \RR^n_+$ and $z \in S$.
We construct $x,y \in \RR^n_+$ such that $\ln x - \ln y = \lambda \mal z$ 
and $x - y = z$ as follows:
if $z_i \neq 0$, then $e^{\lambda_i z_i} \neq 1$, so we may define 
$y_i := z_i / (\e^{\lambda_i z_i} - 1)$ and $x_i := y_i \e^{\lambda_i z_i}$;
otherwise, set $x_i=y_i=1$.
This proves $\Sigma(S) \subseteq \Lambda(S)$. 
\end{proof}

The construction of $x,y$ such that $x-y=z \in S$ in the proof of Lemma~\ref{lem:lambda} 
can be traced back at least to \cite[Section~7]{Feinberg1988}.
See also \cite[Lemma~1]{MAPK} and~\cite[Theorem~5.5]{TSS}.

\begin{lemma} \label{lem:SB}
For $B \in \RR^{r \times n}$ and $S\subseteq \RR^n$, let
\begin{equation} \label{eq:SB}
S_B := \{ x^B - y^B \st x,y \in \RR^n_+ \textrm{ and } x - y \in S^* \} .
\end{equation}
Then, $\sigma(S_B) = \sigma(B(\Sigma(S^*))).$
\end{lemma}
\begin{proof}
For $x,y \in \RR^n_+$, we have $\sigma(x^B - y^B) = \sigma(B(\ln x - \ln y))$ by the strict monotonicity of the logarithm,
and hence
\begin{align*}
\sigma(S_B)
= \sigma\big(\{ B(\ln x - \ln y) \st x,y \in \RR^n_+ \textrm{ and } x - y \in S^* \}\big)
= \sigma(B(\Lambda(S^*))),
\end{align*}
using~\eqref{eq:Lambda}.
By Lemma~\ref{lem:lambda}, $\sigma(S_B) = \sigma(B(\Sigma(S^*)))$.
\end{proof}

\begin{proposition} \label{pro:mono}
Let $B \in \RR^{r \times n}$ and $S \subseteq \RR^n$.  
Further, let $\varphi_B \colon \RR^n_+ \to \RR^r_+$ be the generalized monomial map $\varphi_B(x) = x^B$.
The following statements are equivalent:
\begin{enumerate}[(i)]
\item $\varphi_B$ is injective with respect to $S$.
\item $\sigma(\ker(B)) \cap \sigma(S^*) = \emptyset$.
\end{enumerate}
\end{proposition}

\begin{proof}
By \eqref{eq:SB}, statement \refi is equivalent to $0 \notin S_B$.
By Lemma~\ref{lem:SB}, this is in turn equivalent to $0 \notin B(\Sigma(S^*))$,
that is, $\ker(B)\cap \Sigma(S^*)=\emptyset$.
By \eqref{sigmaSigma}, this is equivalent to statement \refii.
\end{proof}

Comparing Proposition~\ref{pro:mono} with Lemma~\ref{lem:lin},
we observe that $\varphi_B$ being injective with respect to $S$
is equivalent to $\ker(B_\lambda) \cap S^* = \emptyset$, for all $\lambda \in \RR^n_+$.
In case $S$ is a vector subspace,
then $\varphi_B$ is injective on the intersection $(x+S) \cap \RR^n_+$ of any coset of $S$ with the domain $\mathbb{R}^n_+$
if and only if $B_\lambda$ is injective on $S$ for all $\lambda \in \RR^n_+$.

Next we prove the equivalence of statements \refinj and \refsig in Theorem~\ref{thm:main}.

\begin{proposition} \label{pro:poly}
Let $f_\kappa \colon \RR^n_+ \to \RR^m$ be the generalized polynomial map $f_\kappa(x) = A_\kappa \, x^B$,
where $A \in \RR^{m \times r}$, $B \in \RR^{r \times n}$, and $\kappa \in \RR^r_+$.
Further, let $S \subseteq \RR^n$.
The following statements are equivalent:
\begin{itemize}
\item[(inj)]
$f_{\kappa}$ is injective with respect to $S$, for all $\kappa \in \RR^r_+$.
\item[(sig)]
$\sigma(\ker (A)) \cap \sigma(B(\Sigma(S^*))) =\emptyset$.
\end{itemize}
\end{proposition}

\begin{proof}
Statement \refinj  asserts that for $x,y \in \RR^{n}_+$ with $x-y \in S^*$, 
we have $A_\kappa \, (x^B-y^B) \neq 0$ for all $\kappa \in \RR^r_+$.
This is equivalent to asserting that 
$\ker(A_\kappa) \cap S_B = \emptyset$ for all $\kappa \in \RR^r_+$, with $S_B$ as in~\eqref{eq:SB}.
By applying Lemma~\ref{lem:lin} to the matrix $A$ and the subset $S_B$,
this is in turn equivalent to $\sigma(\ker(A)) \cap \sigma(S_B) = \emptyset$.
By Lemma~\ref{lem:SB}, $\sigma(S_B) = \sigma(B(\Sigma(S^*)))$, and the equivalence to statement \refsig is proven.
\end{proof}

A necessary condition for \refsig to hold is $\ker(B) \cap \Sigma(S^*) = \emptyset$
or, equivalently, $\sigma(\ker(B))\cap  \sigma(S^*) = \emptyset$.
By Proposition~\ref{pro:mono},
this corresponds to the fact that
for $f_\kappa$ to be injective with respect to $S$ for all $\kappa \in \RR^r_+$,
the monomial map $\varphi_B$ must be injective with respect to $S$.

To prove the equivalence of statements \reflin and \refjac in Theorem~\ref{thm:main},
we will use the following observation. 

\begin{lemma} \label{lem:jac}
Let $A=(a_{ij})\in\RR^{m \times r}$, $B=(b_{ij}) \in \RR^{r \times n}$,
 $\kappa \in \RR^r_+$, 
and
$\lambda \in \RR^n_+$.
Further,
let $f_\kappa \colon \RR^n_+ \to \RR^m$ be the generalized polynomial map $f_\kappa(x) = A_\kappa \, x^B$.
Then, the sets of all Jacobian matrices $J_{f_\kappa}(x)$ and all matrices $A_\kappa B_\lambda$ coincide:
\[
\left\{J_{f_\kappa}(x)  \st \kappa \in \RR^r_+ \textrm{ and } x \in \RR^n_+ \right\}
=
\left\{ A_\kappa B_\lambda  \st \kappa \in \RR^r_+ \textrm{ and } \lambda \in \RR^n_+\right\} .
\]
\end{lemma}

\begin{proof} 
As $f_{\kappa,i}(x) = \sum_{j=1}^r a_{ij} \, \kappa_j \, x^{b_{j}}$,
the $(i,\ell)$th entry of the Jacobian matrix of $f_{\kappa}$ amounts to
\[
J_{f_\kappa}(x) _{i,\ell} = \DD{f_{\kappa,i}(x)}{x_\ell} =
\sum_{j=1}^r a_{ij} \, \kappa_j \, x^{b_{j}} \, b_{j\ell} \, x_\ell^{-1} .
\]
That is, 
$$J_{f_\kappa}(x)= A \diag(\kappa \mal x^B) B \diag(x^{-1}) = A_{\kappa'} B_\lambda$$
with $\kappa' = \kappa \mal x^B$ and $\lambda = x^{-1}$.  
Clearly, quantifying over all $\kappa \in \RR_+^r$ and $x \in \RR^n_+$ is equivalent to quantifying 
over all $\kappa' \in \RR^r_+$ and $\lambda \in \RR_+^n$.
\end{proof}

We can now combine all the results in this section in the proof of our main theorem.

\begin{proof}[Proof of Theorem~\ref{thm:main}]
The equivalences \reflin $\Leftrightarrow$ \refsig and \refinj $\Leftrightarrow$ \refsig
are shown in Propositions~\ref{pro:lin} and \ref{pro:poly}, respectively.
The equivalence \refjac $\Leftrightarrow$ \reflin follows from Lemma~\ref{lem:jac}.
\end{proof}

In case $S$ is a vector subspace,
the injectivity of $f_\kappa$ (on cosets $x+S$)
can be directly related to the injectivity of the Jacobian of $f_\kappa$ (on $S$). 
This line of thought underpins the original injectivity results on chemical reaction networks due to Craciun and Feinberg~\cite{ME_I} and their extensions. 
In particular,
the case $S=\im(A)$ and $m=n$ arises in applications to chemical reaction networks,
which we address in Section~\ref{sec:app}.  

As discussed in the introduction, a direct corollary of Theorem~\ref{thm:main} characterizes unrestricted injectivity,
that is, the case $S=\RR^n$. See also~\cite[Theorem~3.6]{MR}.

\begin{corollary} \label{cor:unrestricted}
Let $f_\kappa \colon \RR^n_+ \to \RR^m$ be the generalized polynomial map $f_\kappa(x) = A_\kappa \, x^B$,
where $A \in \RR^{m \times r}$, $B \in \RR^{r \times n}$, and $\kappa \in \RR^r_+$.
The following statements are equivalent:
\begin{enumerate}[(i)]
\item
$f_{\kappa}$ is injective, for all $\kappa \in \RR^r_+$.
\item
$\ker(B) = \{0\}$ and $\sigma(\ker(A)) \cap \sigma(\im(B)) = \{0\}$.
\end{enumerate}
\end{corollary}
\begin{proof} Let $S = \RR^n$ and hence $\Sigma(S^*) = S^*$.
By Theorem~\ref{thm:main}, statement \refi is equivalent to
\[
\sigma(\ker (A)) \cap \sigma(B(S^*)) = \emptyset .
\]
Clearly, the above equality does not hold if $\ker(B) \neq \{0\}$. 
If $\ker(B) = \{0\}$, then $B(S^*) = B(S)^* = \im(B)^*$.
Hence, statement \refi is equivalent to $\ker(B) = \{0\}$ and $\sigma(\ker (A)) \cap \sigma(\im(B)^*) = \emptyset$,
which is in turn equivalent to statement \refii.
\end{proof}

The results presented so far concern the injectivity of maps defined on the positive orthant.
In fact, the domain of $f_\kappa(x) = A_\kappa \, x^B$ can be extended to include certain points on the boundary of $\RR^n_+$, and our next result concerns this setting. 
Given $B=(b_{ij}) \in \RR^{r \times n}$, let $\Omega_B \subseteq \overline{\RR}^n_+$ 
be the maximal subset on which the monomial map $\varphi_B(x) = x^B$ is well defined, that is,
\[
\Omega_B := \{ x \in \overline{\RR}^n _+ \st x_j \neq 0 \text{ if } b_{ij} < 0 \text{ for some } i \in [r] \},
\]
and let $\bar{f}_\kappa$ be the extension of $f_\kappa$ to $\Omega_B$.
As it was shown in the context of chemical reaction networks \cite{FeliuWiuf_MAK,WiufFeliu_powerlaw,ShinarFeinberg2012},
injectivity of $f_\kappa$ with respect to $S$ 
precludes the existence of distinct $x,y \in \Omega_B$
in the same coset of $S$ that have the same image under $\bar{f}_{\kappa}$,
i.e.\ with $x-y \in S$ and $\bar{f}_\kappa(x) = \bar{f}_\kappa(y)$.

The technical condition in Proposition~\ref{pro:boundary} below is satisfied
if at least one of the two vectors $x$ and $y$ is in the positive orthant,
or if both contain some zero coordinates, but no coordinate of $x^B$ and $y^B$ vanishes simultaneously.
In particular, if $f_\kappa$ is injective with respect to $S$ for all $\kappa\in \RR^r_+$,
then a coset of $S$ cannot contain a vector in the interior of the positive orthant and a vector on the boundary
that have the same image under $\bar{f}_\kappa$.

\begin{proposition} \label{pro:boundary}
Let $f_\kappa \colon \RR^n_+ \to \RR^m$ be a generalized polynomial map $f_\kappa(x) = A_\kappa \, x^B$,
where $A \in \RR^{m \times r}$, $B \in \RR^{r \times n}$, and $\kappa \in \RR^r_+$. 
Assume that $f_\kappa$ is injective with respect to $S \subseteq \RR^n$, for all $\kappa \in \RR^r_+$.
As above, let $\bar{f}_\kappa$ denote the extension of $f_\kappa$ to $\Omega_B$.
Consider $x,y \in \Omega_B$ with $x \neq y$ and $x-y\in S$, satisfying the following condition:
for any $j \in [r]$, $x^{b_{j}}=y^{b_{j}}=0$ implies that $x_i = y_i = 0$ for all $i \in [n]$ with $b_{ji} \neq 0$.
Then, $\bar{f}_\kappa(x) \neq \bar{f}_\kappa(y)$ for all $\kappa \in \RR^r_+$.
\end{proposition}

\begin{proof} 
For $\ve \in \RR_+$, we define positive vectors $x_\ve, y_\ve \in \RR^n_+$ 
coordinate-wise as follows:
$(x_\ve)_i = x_i + \ve$ and $(y_\ve)_i = y_i + \ve$ whenever $x_i y_i=0$, 
and $(x_\ve)_i = x_i$ and $(y_\ve)_i = y_i$ otherwise. 
Clearly, $x_\ve - y_\ve = x - y \in S$.
We claim that we can choose $\ve$ small enough such that
\[
\sigma(x^B_\ve - y^B_\ve) = \sigma(x^B-y^B).
\]
If $x^{b_j}\neq y^{b_j}$, then clearly $\sign(x^{b_j}_\ve - y^{b_j}_\ve) = \sign(x^{b_j}-y^{b_j})$ for 
sufficiently small $\ve$
since the map $\ve \mapsto x_\ve^B - y_\ve^B$ is continuous. 
Thus it suffices to show that $x^{b_{j}} = y^{b_{j}}$ implies $x_\ve^{b_{j}} = y_\ve^{b_{j}}$.  
In fact, we only need to consider the case when
$x_\ell y_\ell=0$ for some $\ell\in [n]$ with $b_{j\ell} \neq 0$.
Then, our hypothesis implies that $x_i=y_i=0$ for all $i\in [n]$ with $b_{ji} \neq 0$.
By construction, $(x_\ve)_i = (y_\ve)_i = \ve$ for all such $i$
and thus $x_\ve^{b_{j}} = y_\ve^{b_{j}}$, as claimed.

Suppose $\bar{f}_\kappa(x) - \bar{f}_\kappa(y) = A_\kappa (x^B - y^B)= 0$ for some $\kappa \in \RR^r_+$.
Since $x^B-y^B = \lambda \mal (x^B_\ve - y^B_\ve)$ for some $\lambda \in \RR^r_+$,
we obtain
$0 = A_\kappa (x^B - y^B)
= A_{\kappa'} (x^B_\ve - y^B_\ve)
= f_{\kappa'} (x_\ve) - f_{\kappa'} (y_\ve)$, where $\kappa' = \kappa \mal \lambda$.
Clearly, this contradicts the hypothesis that $f_{\kappa'}$ is injective with respect to $S$.
\end{proof}

A related result concerning injectivity up to the boundary in the two-dimensional case appears in~\cite{SottileZhu}.

\subsection{Determinantal conditions} \label{subsec:det}

In this subsection, we characterize the injectivity of a family of maps on the positive orthant $f_\kappa \colon \RR^n_+ \to \RR^n$, $x \mapsto A_\kappa \, x^B$,
with respect to $S \subseteq \RR^n$, 
in the case where $S$ is a vector subspace with $\dim(S) = \rank(A)$.
In particular, we provide injectivity conditions
in terms of determinants and signs of maximal minors. 

Given a proper vector subspace $S\subseteq \RR^n$ of dimension $s$, it can be presented as the image of a full-rank matrix $C\in \RR^{n \times s}$,
or as the kernel of 
a full-rank matrix $Z\in \RR^{(n-s)\times n}$,  whose
rows are a basis of $S^\perp$. To recall the relation between the maximal minors of $C$ and $Z$, we need the following
notation.
For $n\in \NN$ and a subset $I=\{i_1, \dots, i_s\}\subseteq [n]$, let $I^c= \{j_1, \dots, j_{n-s}\}$ be the complement of $I$ in $[n]$.
For $i_1 < \dots < i_s$ and $j_1 < \dots < j_{n-s}$, let $\tau(I) \in \{ \pm 1 \}$ denote the sign of the permutation that sends $1,\dots, n$ to $j_1, \dots, j_{n-s}, i_1, \dots, i_s$, respectively.

\begin{lemma}\label{lem:Gale1} 
Let $s,n$ be natural numbers with $0<s< n$ and $C\in \RR^{n \times s}, Z \in \RR^{(n-s)\times n}$ full-rank matrices with $\im(C)=\ker(Z)$.
Then, there exists a nonzero real number $\delta$ such that
\begin{equation*}\label{eq:detM3}
\delta\det(C_{I,[s]})  =  (-1)^{\tau( I)} \det(Z_{[n-s],I^c}), 
\end{equation*}
for all subsets $I\subseteq[n]$ of cardinality $s$.
\end{lemma}
Lemma~\ref{lem:Gale1} is well known 
(see for instance, \cite[p.~94, Equation~(1.6)]{gkzbook} and \cite[Appendix A]{gkzbook} on the determinant of a complex, in particular, the proofs of Lemma~5 and Proposition~11 or Theorem 12.16 in~\cite{JoswigTheobald2013}). 
The full-rank matrices $Z, C$ are called \emph{Gale dual}; see Definition~\ref{def:Gale} below.

Let $s\leq n$. For $A'\in\RR^{s \times r}$, $B \in \RR^{r \times n}$, and $Z\in \RR^{(n-s)\times n}$, let $\Gamma_{\kappa,\lambda}\in \RR^{n\times n}$ be the square matrix
given in block form as
\begin{align} \label{eq:mkappalambda}
\Gamma_{\kappa,\lambda} =
\begin{pmatrix}
Z \\ 
A'_\kappa B_\lambda
\end{pmatrix}, \qquad \textrm{for } \kappa\in \RR^r_+ ~{\rm and }
~ \lambda\in \RR^n_+.
\end{align}
For simplicity, we do not treat the case $s=n$ separately.
Instead, we use $\Gamma_{\kappa,\lambda} = A'_\kappa B_\lambda$ and $\det(Z_{[n-s],I^c})=1$ in the statements below for this case.

We start with two useful lemmas.
\begin{lemma}\label{lem:prepdet2}
Let $\Gamma_{\kappa,\lambda}$ be the matrix defined in \eqref{eq:mkappalambda},
for $s\leq n$,  $A'\in\RR^{s \times r}$, $B \in \RR^{r \times n}$, $Z\in \RR^{(n-s)\times n}$, $\kappa\in \RR^r_+$,
and $\lambda\in \RR^n_+$. Then
\[
\det(\Gamma_{\kappa,\lambda}) = \sum\nolimits_{I,J}   (-1)^{\tau(I)} \det(Z_{[n-s],I^c}) \det(A'_{[s],J}) \det(B_{J,I}) \kappa^J \lambda^I, 
\]
where we sum over all subsets $I\subseteq [n]$, $J \subseteq [r]$  of cardinality $s$, and $\kappa^J = \prod_{j \in J} \kappa_j$, $\lambda^I = \prod_{i \in I} \lambda_i$. 
\end{lemma}
\begin{proof}
By Laplace expansion on the bottom $s$ rows of $\Gamma_{\kappa,\lambda}$,  we have that
\begin{equation*}\label{eq:detM}
\det(\Gamma_{\kappa,\lambda})  =  \sum\nolimits_{I}  (-1)^{\tau(I)} \det(Z_{[n-s],I^c})\det((A'_\kappa  B_\lambda)_{[s],I}),
\end{equation*}
where we sum over all subsets $I \subseteq [n]$ of cardinality $s$. 
The Cauchy-Binet formula yields
\begin{equation*}
\det((A'_\kappa  B_\lambda)_{[s],I}) = \sum\nolimits_J \det((A'_\kappa)_{[s],J}) \det((B_\lambda)_{J,I})
= \sum\nolimits_J \det(A'_{[s],J}) \det(B_{J,I}) \kappa^J \lambda^I, 
\end{equation*}
where we sum over all subsets $J \subseteq [r]$ of cardinality $s$.
\end{proof}

\begin{lemma}\label{lem:prepdet1}
Let $q(c)\in \RR[c_1,\dots,c_\ell]$ be a nonzero homogeneous polynomial, with degree at most one in each variable. There exists $c^*\in \RR_+^\ell$ such that $q(c^*) =0$ if and only if $q(c)$ has both positive and negative coefficients.
\end{lemma}
\begin{proof}
If all coefficients of $q(c)$ have the same sign, it is clear  that $q(c^*)\neq 0$, for all $c^*\in \RR_+^\ell$.
To prove the reverse implication, 
let $\alpha \, c^v$ be any monomial of $q$ (so, $v\in \{0,1\}^\ell$). For $\epsilon\in \RR_+$, define 
$c(\epsilon) \in  \RR_+^\ell$ by 
$c_i(\epsilon):=\epsilon$ if $v_i=1$ and $c_i(\epsilon):=1$ if $v_i=0$. 
Then $q(c(\epsilon))$ is a univariate polynomial in $\epsilon$ of the same degree as $q$ 
and with leading coefficient $\alpha$.
For sufficiently large $\epsilon$, 
the sign of $q(c(\epsilon))$ is the sign of $\alpha$.
Therefore, if two nonzero coefficients have opposite signs, $q(c)$ takes both  positive and negative values, and
so by continuity, there exists $c^*\in \RR_+^\ell$  such that $q(c^*)=0$.
\end{proof}

The following result generalizes \cite[Proposition~5.2--5.3]{WiufFeliu_powerlaw}.

\begin{theorem} \label{thm:jacdet}
Let $f_\kappa \colon \RR^n_+ \to \RR^m$ be the generalized polynomial map $f_\kappa(x) = A_\kappa \, x^B$,
where $A \in \RR^{m \times r}$, $B \in \RR^{r \times n}$, and $\kappa \in \RR^r_+$.

Assume that $\rank(A)=s$, and consider a vector subspace $S \subseteq \RR^n$ with $\dim(S) = s$.
Let $Z\in \RR^{(n-s)\times n}$ and 
 $C\in \RR^{n \times s}$ be matrices presenting $S$, that is, such that $\im(C)=S=\ker(Z)$.
Given $A'\in \RR^{s \times r}$ with $\ker(A)=\ker(A')$,  call
$\widetilde{A} = C A' \in \RR^{n \times r}$, and let
$\Gamma_{\kappa,\lambda}\in \RR^{n \times n}$ be the square matrix
associated to $A',B,Z$,   $\kappa \in \RR^r_+$, and $\lambda \in \RR^n_+$ as in \eqref{eq:mkappalambda}.

The following statements are equivalent:
\begin{itemize}
\item[(inj)]
$f_\kappa$ is injective with respect to $S$, for all $\kappa \in \RR^r_+$.
\item[(det)] 
Viewed as a polynomial in $\kappa$ and $\lambda$, $\det(\Gamma_{\kappa,\lambda})$ is nonzero and all of its nonzero coefficients have the same sign.
\item[(min)]
For all subsets $I\subseteq [n]$, $J\subseteq [r]$ of cardinality $s$,
the product $\det(\widetilde{A}_{I,J}) \det(B_{J,I})$ either is zero or has the same sign as all other nonzero products,
and moreover, at least one such product is nonzero.
\end{itemize}
\end{theorem}

\begin{proof}
Using the equivalence \refinj $\Leftrightarrow$ \reflin of Theorem~\ref{thm:main} and that $S$ is the solution set to the equation $Zx=0$,
statement \refinj  is equivalent to $\Gamma_{\kappa,\lambda}(x)\neq 0$ for all $\kappa\in \RR^r_+$, $\lambda \in \RR^n_+$, and $x\in \RR^n$ with $x\neq 0$.
As $\Gamma_{\kappa,\lambda}$ is a square matrix, this is in turn equivalent to $\det(\Gamma_{\kappa,\lambda})\neq 0$,
for all $\kappa \in \RR^r_+$ and $\lambda \in \RR^n_+$. 
By Lemma~\ref{lem:prepdet2}, $\det(\Gamma_{\kappa,\lambda})$ is a  homogeneous polynomial in $\kappa,\lambda$
with degree at most one in each variable.
Hence the equivalence \refinj $\Leftrightarrow$ \refdet follows from Lemma~\ref{lem:prepdet1}.

By Cauchy-Binet, $\det(\widetilde{A}_{I,J}) =\det(C_{I,[s]})\det(A'_{[s],J})$, and hence the equivalence \refdet $\Leftrightarrow$ \refmin follows from Lemmas~\ref{lem:Gale1}  and \ref{lem:prepdet2}.
\end{proof}

Therefore,  injectivity of $A_\kappa x^B$ 
can be assessed by computing either 
the nonzero products of the $s\times s$ minors of $\widetilde{A}$ and $B$
or the determinant of the symbolic matrix $\Gamma_{\kappa,\lambda}$.
Further, it follows from Theorem~\ref{thm:jacdet} that $\det(\Gamma_{\kappa,\lambda})$ equals the sum of the principal minors of size $s$ of $\widetilde{A}_\kappa B_\lambda$.
This implies the interesting fact that if $\det(\Gamma_{\kappa,\lambda})$ is nonzero, it equals the product of the nonzero eigenvalues of $\widetilde{A}_\kappa B_\lambda$.

Clearly,  the hypotheses of Theorem~\ref{thm:jacdet} are fulfilled for $S=\im(A)$. 
In this case,  the matrix $C\in \RR^{n \times s}$ can be chosen to satisfy
$A=CA'$. Therefore, we obtain the following corollary, which was proven in \cite{WiufFeliu_powerlaw}.

\begin{corollary} \label{cor:jacdet}
Let $f_\kappa \colon \RR^n_+ \to \RR^n$ be the generalized polynomial map $f_\kappa(x) = A_\kappa \, x^B$,
where $A \in \RR^{n \times r}$, $B \in \RR^{r \times n}$, and $\kappa \in \RR^r_+$.
Further, let $s=\rank(A)$. The following statements are equivalent:
\begin{enumerate}
\item[(inj)]  $f_\kappa$ is injective with respect to $\im(A)$, for all $\kappa \in \RR^r_+$.
\item[(min)] For all subsets $I\subseteq [n]$, $J\subseteq [r]$ of cardinality $s$,
the product $\det(A_{I,J}) \det(B_{J,I})$ either is zero or has the same sign as all other nonzero such products, 
and moreover, at least one such product is nonzero.
\end{enumerate} 
\end{corollary}

Let $A \in \RR^{m \times r}$ and $B \in \RR^{r \times n}$ have full rank $m$ and $n$, respectively.
By Corollary~\ref{cor:unrestricted},
the (unrestricted) injectivity of $f_\kappa(x) = A_\kappa \, x^B$ is equivalent to $\sigma(\ker(A)) \cap \sigma(\im(B)) = \{0\}$.
For $m<n$, the intersection $\ker(A) \cap \im(B)$ is always nontrivial since $(r-m)+n>r$.
For $m=n$, determinantal conditions are given in Corollary~\ref{cor:det} below;
see also \cite[Theorem~3.1]{Chaiken1996} and \cite[Corollary~8]{CGS}.
For $m>n$, the problem is NP-complete; see Section~\ref{sec:algo}.

\begin{corollary} \label{cor:det}
Let $A \in \RR^{n \times r}$ and $B \in \RR^{r \times n}$ be matrices of rank $n$.
The following statements are equivalent:
\begin{enumerate}[(i)]
\item $\sigma(\ker(A)) \cap \sigma(\im(B)) = \{0\}$.
\item For all subsets $J\subseteq [r]$ of cardinality $n$, 
the product  $\det(A_{ [n] ,J})\det(B_{J, [n] })$
either is zero or has the same sign as all other nonzero products,
and moreover, at least one such product is nonzero.
\end{enumerate}
\end{corollary}
\begin{proof}
The (unrestricted) injectivity of $f_\kappa(x) = A_\kappa \, x^B$ for all $\kappa \in \RR^r_+$
is equivalent to both \refi by Corollary~\ref{cor:unrestricted} (since $\ker(B)=\{0\}$)
and \refii by Corollary~\ref{cor:jacdet} (since $\im(A)=\RR^n$).
\end{proof}

\section{Applications} \label{sec:app}

The first application of our results is to study steady states (or equilibria) of dynamical systems
induced by generalized polynomial maps.
In Subsection~\ref{subsec:power-law},
we introduce such \emph{power-law systems} and state our results in this setting.
In Subsection~\ref{subsec:special},
we give sign conditions that preclude/guarantee the existence of multiple steady states of a particular form.
In Subsection~\ref{subsec:solve}, we show how our results reveal the first partial multivariate generalization
of Descartes' rule of signs in real algebraic geometry and interpret our results in the language of oriented matroids.

\subsection{Power-law systems} \label{subsec:power-law}

Power-law systems arise naturally as models of systems of interacting species, such as chemical reaction networks.
Other examples include the classical Lotka-Volterra model in ecology~\cite{Murray2002}
and the SIR model in epidemiology~\cite{AndersonMay1991}.

For readers unfamiliar with chemical reaction networks,
we elaborate on the construction of the corresponding dynamical systems.
A chemical reaction network consists of a set of $n$ molecular species and a set of $r$ reactions,
where the  left- and right-hand sides of the reactions are formal sums of species,
called reactant and product complexes, respectively.
A {\em kinetic system} describes the dynamics of the species concentrations $x$,
where each reaction contributes to the dynamics an additive term:
namely, a corresponding reaction vector (the difference between the product and reactant complexes)
multiplied by a particular reaction rate (a nonnegative function of the concentrations, called kinetics).
Thus, a kinetic system has the form 
\begin{equation} \label{eq:kinetic_system}
\dd{x}{t} =  N K(x) ,
\end{equation} 
where the columns of the {\em stoichiometric matrix} $N$ are the reaction vectors
and the $i$th coordinate of $K(x)$ is the rate function of the $i$th reaction.
The right-hand side of~\eqref{eq:kinetic_system} is called the {\em species-formation rate function}.  
In power-law systems,
the kinetics are given by monomials with real exponents~\cite{Savageau}.  
More precisely, the {\em power-law system} arising from the stoichiometric matrix $N \in \RR^{n \times r}$,
a {\em kinetic-order matrix} $V \in \RR^{r \times n}$,
and {\em rate constants} $\kappa \in \RR^r_+$
is the kinetic system~\eqref{eq:kinetic_system} with kinetics $K(x)=\kappa \mal x^V$.
That is, the species-formation rate function $f_\kappa \colon \RR^n_+ \to \RR^n$ is given by
\begin{equation} \label{eq:sfr}
f_\kappa(x) = N_\kappa \, x^V .
\end{equation}
In fact, the domain of $f_\kappa$ may be extended to $\Omega_V \subseteq \overline{\RR}^n_+$,
the maximal subset on which the monomial map $\varphi_V \colon x \mapsto x^V$ is well defined.
We note that, without further restrictions on the matrix $V$,
a power-law system may exhibit physically/chemically meaningless behavior.
For example,
a trajectory starting in the interior may reach the boundary of the positive orthant in finite time with nonzero velocity.

The vector subspace $S=\im(N)$ is the \emph{stoichiometric subspace},
and the sets $(x'+S) \cap \RR^n_+$ for $x' \in \RR^n_+$ are the positive \emph{compatibility classes}.
As explained in the introduction, 
a trajectory starting at a point $x' \in \RR^n_+$
is confined to the coset $x'+S$.
As a consequence,
we study power-law systems 
restricted to compatibility classes. 
In particular, we want to characterize
whether there exist distinct $x,y \in \RR^n_+$ such that $x-y \in S$ and $f_\kappa(x)=f_\kappa(y)=0$ for some $\kappa \in \RR^r_+$.
In our terminology, 
if $f_\kappa$ is injective with respect to $S$ for all $\kappa \in \RR^r_+$,
then no such $x,y$ can exist,
that is, multiple {\em positive} steady states cannot occur within one compatibility class
for any choice of the rate constants.

\begin{example}[Mass-action systems] \label{ex:mass-action}
\emph{Mass-action systems} form a family of power-law systems, and they are 
widely used to model the dynamics of chemical reaction networks. 
In mass-action systems,
the rate of a chemical reaction is a monomial in the concentrations of the reactant species;
more precisely, the exponents of the concentrations are the corresponding {\em stoichiometric coefficients}, i.e., the coefficients of the species in the reactant complex.
As a consequence,
the kinetic-order matrix $V$ is a nonnegative integer matrix,
which encodes for each reaction the stoichiometries of the reactant species,
and the map $f_\kappa(x)$ is a polynomial map in the standard sense with domain $\Omega_V = \overline{\RR}_+^n$. 
Mass-action systems are at the core of the so-called {\em chemical reaction network theory}, initiated by Horn, Jackson, and Feinberg in the 1970s \cite{Feinberg1972,Horn1972,HornJackson1972};
see also the surveys~\cite{feinbergnotes,gu03}.
\end{example}

\begin{example}[Generalized mass-action systems]
The {\em law of mass-action}, proposed by Guldberg and Waage in the 19th century \cite{gw64},
refers to both the formula for chemical equilibrium, which holds for all reactions,
and the formula for the reaction rate (explained in Example~\ref{ex:mass-action}),
which holds only for elementary reactions in homogeneous and dilute solutions.
To model the dynamics of chemical reaction networks in more general environments,
power-law kinetics has been considered under different formalisms \cite{HornJackson1972, Savageau}.
The notion of \emph{generalized mass-action systems} as introduced in \cite{MR,MR14} is a direct extension of mass-action systems, in particular, it includes the inherent structure of chemical reaction networks.
\end{example}

\begin{example}[S-systems]
\emph{S-systems} form another family of power-law systems.
This research area was initiated by the work of Savageau in the late 1960s~\cite{Savageau}.
In S-systems,
the formation rate of each species consists of one production term and one degradation term.
In other words, the components $f_{\kappa,i}(x)$ are binomial,
and each row of the stoichiometric matrix $N$ contains the entries $1$ and $-1$, and all other entries are zero.
S-systems can be used to infer gene regulatory networks, for instance, if the regulation logic is not known or the precise mechanisms are inaccessible.
Further, many common kinetic systems, including
(generalized) mass-action systems,
 can be approximated by S-systems
after a process called recasting \cite{Savageau1987}.
\end{example}

An injectivity criterion for precluding multistationarity in fully open networks with mass-action kinetics
was introduced by Craciun and Feinberg~\cite{ME_I}
and has been extended in various ways~\cite{ME_entrapped,ME_semiopen,BanajiDonnell,gnacadja_linalg,FeliuWiuf_MAK,WiufFeliu_powerlaw}.
Our contribution to this topic builds on these results and is summarized in Theorem~\ref{thm:chem}.
It is a restatement of Theorem~\ref{thm:main} in the setting of power-law systems;
in particular, $m=n$ and $S = \im(N)$ is a vector subspace.
In this case, Corollary~\ref{cor:jacdet} allows us to add the condition~\refmin.
Further,
the condition~\refinj  concerns the injectivity of the generalized polynomial map {\em on compatibility classes},
and \refjac addresses the injectivity of the Jacobian matrix {\em on the stoichiometric subspace}.

\begin{theorem}[Theorem~\ref{thm:main} for power-law systems] \label{thm:chem}
Let $f_\kappa \colon \RR^n_+ \to \RR^n$ be the species-formation rate function $f_\kappa(x) = N_\kappa \, x^V$ of a power-law system,
where $N \in \RR^{n \times r}$, $V \in \RR^{r \times n}$, and $\kappa \in \RR^r_+$.
Further, let $S = \im(N)$ and $s = \rank(N)$.
The following statements are equivalent:
\begin{enumerate}
\item[(inj)]
$f_\kappa$ is injective on every compatibility class, for all $\kappa \in \RR^r_+$.
\item[(jac)]
The Jacobian matrix $J_{f_\kappa}(x)$ is injective on the stoichiometric subspace $S$,
for all $\kappa \in \RR^r_+$ and $x \in \RR^n_+$.
\item[(min)] For all subsets $I\subseteq [n]$, $J\subseteq [r]$ of cardinality $s$,
the product $\det(N_{I,J}) \det(V_{J,I})$ either is zero or has the same sign as all other nonzero such products, 
and moreover, at least one such product is nonzero.
\item[(sig)]
$\sigma(\ker(N)) \cap \sigma(V(\Sigma(S^*))) = \emptyset$.
\end{enumerate}
\end{theorem}
If the conditions of Theorem~\ref{thm:chem} hold, then multistationarity is precluded.
In the context of chemical reaction networks,
the equivalence of conditions \refinj, \refjac, and \refmin was proven in \cite{WiufFeliu_powerlaw}.
Thus, our contribution is condition \refsig.
 
\begin{remark}
Injectivity results for generalized polynomial maps also preclude multistationarity
for \emph{strictly monotonic} kinetics~\cite{WiufFeliu_powerlaw,feliu-bioinfo},
which include power-law kinetics.
In the study of {\em concordant} networks~\cite{ShinarFeinberg2012},
sign conditions preclude multistationarity for {\em weakly monotonic} kinetics.
Injectivity results for differentiable maps and various classes of kinetics using P-matrices
appear in \cite{BanajiDonnell,BanajiCraciun2009,BanajiCraciun2010,banaji_pantea,Gale:1965p474}.
P-matrices are defined by the positivity of principal minors, which is related to condition \refmin in this work. 
Analysis of the signs of minors of Jacobian matrices with applications to counting steady states appear in \cite{CHW08,HeltonDeterminant,HKK}.
\end{remark}

\subsection{Precluding/guaranteeing special steady states} \label{subsec:special}

In this subsection, we relate results on injectivity and sign vectors
occurring in the chemical reaction literature for ``special'' steady states, under seemingly different hypotheses.
On one side,
we study {\em complex balancing equilibria} defined for mass-action systems~\cite{Feinberg1972,Horn1972,HornJackson1972}
and extended to generalized mass-action systems~\cite{MR};
on the other side, we consider {\em toric steady states}~\cite{TSS}.
The common feature of all these cases is
that the steady states under consideration lie in a generalized variety
that has dual equivalent presentations: via generalized binomial equations and via a generalized monomial parametrization. 
Our results give conditions for precluding multiple special steady states (Proposition~\ref{pro:dual})
and for guaranteeing multiple special steady states (Corollary~\ref{cor:multistat}). 

Given  $M \in  \RR^{d' \times n}$ and $x^* \in \RR^n_+$, we denote the corresponding fiber of
$x \mapsto x^M$ by
\begin{equation*} 
Z^M_{x^*} := \left\{ x \in \RR^n_+ \st x^{M} = (x^*)^{M} \right\} .
\end{equation*}
We note that in the literature on chemical reaction network theory,
the alternate formulation $Z^M_{x^*} = \{ x \in \RR^n_+ \st \ln(x) - \ln(x^*) \in \ker(M) \}$ is used. 
Also, if we denote by $m_i$ the $i$th  row vector of $M$ and write it as $m_i = m_i^+ - m_i^-$ with
$m_i^+,m_i^- \in \overline{\RR}_+^n$, then  for any positive $\gamma_i$, 
the generalized monomial equation $x^{m_i} = \gamma_i$  is equivalent to the generalized binomial equation $x^{m_i^+} - \gamma_i x^{m_i^-}=0$,
when we restrict our attention to $x \in \RR_+^n$.

\begin{definition} \label{def:Gale}
Two matrices $M \in \RR^{d' \times n}$ and $B \in \RR^{n \times d}$
with $\im(B) = \ker(M)$ and $\ker(B) = \{0\}$  
are called \emph{Gale dual}.
\end{definition}
\noindent In the usual definition of Gale duality, the matrix $M$ is required to have full rank $d'= n-d$.

The following lemma is classic.

\begin{lemma} \label{lem:Gale}
Let $M \in \RR^{d' \times n}$ and $B \in \RR^{n \times d}$ be Gale dual. 
Then, for any $x^* \in \RR^n_+$, the fiber $Z^M_{x^*}$ can be parametrized as follows:
\[
Z^M_{x^*} = \{ x^* \mal \e^v \st v \in \ker(M) \}
= \{ x^* \mal \xi^{B} \st \xi \in \RR^d_+ \} .
\]
\end{lemma}

\begin{proof}
We start by proving the first equality.
We have $x\in Z^M_{x^*}$ if and only if $x^M=(x^*)^M$, which is equivalent to $M \, (\ln x - \ln x^*) = 0$.
Therefore, $x\in Z^M_{x^*}$ if and only if $v := \ln x - \ln x^* \in \ker(M)$, that is, $x = x^* \mal \e^v$ with $v \in \ker(M)$. 
Now we turn to the second equality.
Since the columns of $B$ form a basis for $\ker(M)$,
we can write $v \in \ker(M)$ uniquely as $v = B \, t$ for some $t \in \RR^d$.
By introducing $\xi := \e^t \in \RR^d_+$, we obtain
\[
(\e^v)_i = \e^{v_i} = \e^{\sum_j b_{ij} t_j} = \textstyle{\prod_j \xi_j^{b_{ij}}} = \xi^{b_{i}} = (\xi^B)_i ,
\]
that is, $\e^v = \xi^{B}$, so the inclusion $\subseteq$ holds.  Similarly, $\supseteq$ holds via $v:=B \log \xi$.
\end{proof}

We consider a power-law system~\eqref{eq:sfr}
and assume that the set of steady states
contains the positive part of a generalized variety defined by generalized binomials,
according to the following definition. Recall the connection between certain monomial and binomial equations explained before Definition~\ref{def:Gale}.

\begin{definition}\label{def:sss}
Let $f_\kappa \colon \RR^n_+ \to \RR^n$ be the species-formation rate function $f_\kappa(x) = N_\kappa \, x^V$ of a power-law system,
where $N \in \RR^{n \times r}$, $V \in \RR^{r \times n}$, and $\kappa \in \RR^r_+$.
Further, let $M \in \RR^{d' \times n}$ and $\gamma\colon \RR^r_+\rightarrow  \RR_+^{d'}$.
Consider the family of generalized varieties
\[
Y^{M,\gamma}_\kappa := \left\{x \in \RR^n_+ \st x^M = \gamma(\kappa) \right\} \quad \text{for } \kappa \in \RR_+^r ,
\]
and assume that each such generalized variety consists of steady states of the corresponding power-law system:
\[
Y^{M,\gamma}_\kappa \subseteq \{ x\in \RR^n_+ \st f_\kappa(x) = 0\} \quad \text{for all } \kappa \in \RR^r_+ .
\]
An element $x^* \in Y^{M,\gamma}_{\kappa^*}$ is called a \emph{special steady state for $\kappa^*$}.
\end{definition}

According to the definition,
$x^*$ is a special steady state for $\kappa^*$
if and only if $(x^*)^M=\gamma(\kappa^*)$, or, equivalently, $Y^{M,\gamma}_{\kappa^*} = Z^M_{x^*}$.
Clearly, if $\gamma(\kappa^*)$ does not belong to the image of the monomial map $\varphi_M\colon x\mapsto x^M$,
then $Y^{M,\gamma}_{\kappa^*} =\emptyset$.
As already mentioned,
special steady states include {\em complex balancing equilibria} of generalized mass-action systems~\cite{MR} and {\em toric steady states}~\cite{TSS}.
In both cases, the relevant map $\gamma$ is a rational function.  

Consider $N,$ $V,$ $M$, and $\gamma$ as in Definition~\ref{def:sss}.
Let $x^* \in \RR^n_+$ be a special steady state for $\kappa^*$.
By Lemma~\ref{lem:Gale}, the corresponding set of special steady states $Y^{M,\gamma}_{\kappa^*} = Z^M_{x^*}$
can be parametrized as $\{ x^* \mal \xi^{B} \st \xi \in \RR^d_+ \}$,
where $B\in \RR^{n\times d}$ with $\im(B) = \ker(M)$ and $\ker(B) = \{0\}$.
In fact,
we are interested in the intersection of the set of special steady states with some compatibility class,
\begin{equation*} \label{eq:intersection_special}
Z^M_{x^*} \cap (x'+S).
\end{equation*}
If the intersection is nonempty, then there exist $\xi \in \RR^d_+$ and $u \in S$ such that
\[
x^* \mal \xi^{B} = x' + u,
\]
and multiplication by a matrix $A \in \RR^{m \times n}$ for which $\ker(A) = S$ yields
\[
A \, (x^* \mal \xi^{B}) = A \, x' .
\]
Thus, using $\ker(B)=\{0\}$,
injectivity of the generalized polynomial map $f_{x^*}\colon \RR^d_+ \to \RR^n_+$,
\[
f_{x^*} (\xi) = A \, (x^* \mal \xi^{B}) = A_{x^*} \, \xi^B,
\]
is equivalent to the uniqueness of special steady states  
in every compatibility class.
Therefore, if $f_{x^*}$ is injective for {\em all} $x^* \in \RR_+^n$,
as characterized in Proposition~\ref{pro:dual} below,
then multiple special steady states are precluded for {\em all} rate constants.
Note that Theorem~\ref{thm:chem} precludes multiple ``general'' steady states.

\begin{proposition} \label{pro:dual} 
Let $M \in \RR^{d' \times n}$ and $B \in \RR^{n \times d}$ be Gale dual, 
$S \subseteq \RR^n$ be a vector subspace, and $A \in \RR^{m \times n}$ such that $S = \ker(A)$.
The following statements are equivalent:
\begin{enumerate}[(i)]
\item The monomial map $\varphi_M \colon \RR^n \to \RR^{d'}$, $x \mapsto x^M$ is injective
on $(x'+S) \cap \RR^n_+$, for all $x' \in \RR^n_+$.
\item $\sigma(\ker(M)) \cap \sigma(S) = \{0\}$.
\item The polynomial map $f_{x^*} \colon \RR^d_+ \to \RR^m, \, \xi \mapsto A_{x^*} \, \xi^B$ is injective, for all $x^* \in \RR^n_+$.
\end{enumerate}
\end{proposition}
\begin{proof}
Statement \refii is equivalent to $\sigma(\ker(A)) \cap \sigma(\im(B)) = \{0\}$, by the definitions of the matrices.
\refi $\Leftrightarrow$ \refii holds by Proposition~\ref{pro:mono} (for a vector subspace $S$).
\refiii $\Leftrightarrow$ \refii holds by Corollary~\ref{cor:unrestricted}.
\end{proof}
In other words,
injectivity of monomial maps on cosets of a vector subspace
is equivalent to
injectivity of a related family of polynomial maps on the positive orthant.

\begin{remark}
Related sign conditions for injectivity appear in 
\cite[Lemma~4.1]{Feinberg1995b}, \cite[Lemma~1]{MAPK}, \cite[Theorem~5.5]{TSS}, and \cite[Proposition~3.1 and Theorem~3.6]{MR}.
In generalized mass-action systems \cite{MR},
uniqueness of {\em complex balancing equilibria} is guaranteed by the sign condition $\sigma(S) \cap \sigma(\tilde{S}^\perp) = \{0\}$,
where $\tilde{S}$ is the kinetic-order subspace with $\tilde{S}^\perp = \ker(M) = \im(B)$.
In the specific case of mass-action systems,
the stoichiometric and kinetic-order subspaces coincide, $S = \tilde{S}$,
and hence $\sigma(S) \cap \sigma(S^\perp) = \{0\}$ holds trivially. Further, in this case, if complex balancing equilibria exist, all steady states are of this form~\cite[Theorem 6A]{HornJackson1972}
and multistationarity cannot occur.
The sign condition for precluding multiple {\em toric steady states} \cite{TSS}
takes the form $\sigma(\im(\mathcal{A}^T)) \cap \sigma(\ker(\mathcal{Z}^T)) = \{0\}$,
where we use calligraphic fonts to avoid confusion with symbols in this work.
The matrix $\mathcal{A}$ specifies the parametrization of $Z$,
whereas the matrix $\mathcal{Z}$ defines the stoichiometric subspace~$S$:
$\ker(M) = \im(\mathcal{A}^T)$ and $S = \ker(\mathcal{Z}^T)$.
\end{remark}

We close this subsection by considering the case when statement \refii in Proposition~\ref{pro:dual} does not hold.
In this case, multiple special steady states in one compatibility class are possible, provided that every $x^*\in \RR^n_+$ 
is a special steady state for some $\kappa^*$.

\begin{corollary} \label{cor:multistat}
Let $f_\kappa \colon \RR^n_+ \to \RR^n$ be the species-formation rate function $f_\kappa(x) = N_\kappa \, x^V$ of a power-law system
with stoichiometric subspace $S = \im(N)$,
where $N \in \RR^{n \times r}$, $V \in \RR^{r \times n}$, and $\kappa \in \RR^r_+$.
Further, let $M \in \RR^{d' \times n}$, $\gamma \colon \RR^r_+\rightarrow \RR_+^{d'}$, and $Y^{M,\gamma}_\kappa$ be a set of special state states as in Definition~\ref{def:sss}.
Assume that:
\begin{itemize}
\item[(i)] $\sigma(\ker(M)) \cap \sigma(S) \neq \{0\}$.
\item[(ii)] For all $x \in \RR^n_+$, there exists $\kappa \in \RR^r_+$ such that $x \in Y^{M,\gamma}_{\kappa}$.
\end{itemize}
Then there exist $\kappa^* \in \RR^{r}_+$ and distinct $x^*,y^* \in \RR^n_+$
such that
\[
x^*, y^* \in Y^{M,\gamma}_{\kappa^*} \quad \text{and} \quad x^* - y^* \in S .
\]
In other words, there exist multiple special steady states in some compatibility class.
\end{corollary}

\begin{proof}
Assume $\sigma(\ker(M)) \cap \sigma(S) \neq \{0\}$.
By \refii $\Leftrightarrow$ \refi in Proposition~\ref{pro:dual},
there exist $x^*,y^* \in \RR^n_+$ with $x^* \neq y^*$, $x^*-y^*\in S$, and $(x^*)^M =(y^*)^M$, that is, $x^*,y^*\in Z^M_{x^*}$.
 By assumption \refii, there exists $\kappa^*\in \RR^r_+$ such that $x^* \in Y^{M,\gamma}_{\kappa^*}$, that is,
$Z^M_{x^*} = Y^{M,\gamma}_{\kappa^*}$.
Hence, $x^*,y^*  \in Y^{M,\gamma}_{\kappa^*}$.
\end{proof}

In the case of complex balancing equilibria,
the crucial assumption \refii in Corollary~\ref{cor:multistat}
follows from weak reversibility (cf.~\cite[Lemma~3.3]{MR}).
In the case of toric steady states, 
it is guaranteed by the existence of a positive toric steady state for some $\kappa$
or, equivalently,  by the existence of a positive vector in the kernel of $N$ (cf.~\cite[Theorem~5.5]{TSS}). 

\subsection{Solving systems of generalized polynomial equations} \label{subsec:solve}

In this subsection, we prove the partial multivariate generalization of Descartes' rule, Theorem~\ref{thm:Descartes}.
The bound on the number of positive solutions in statement~\refbound is a direct consequence of Corollaries~\ref{cor:unrestricted} and \ref{cor:det},
and it was proved  
in previous works, e.g.\ in \cite[Corollary~8]{CGS}.
The existence of positive solutions in statement~\refsurj 
relies on the surjectivity result in~\cite[Theorem~3.8]{MR}.
The framework of our results is the theory of oriented matroids,
which is concerned with combinatorial properties of geometric configurations.

\begin{proposition} \label{pro:polyequ}
Let $A \in \RR^{m \times r}$ and $B \in \RR^{r \times n}$ with full rank $n$.
The following statements are equivalent:
\begin{enumerate}[(i)]
\item
For all $\kappa \in \RR^r_+$ and $y \in \RR^m$,
the system of $m$ generalized polynomial equations in $n$ unknowns
\begin{equation*}
\sum_{j =1}^r a_{ij} \, \kappa_j \,  x_1^{b_{j1}} \cdots x_n^{b_{jn}}= y_i, \quad i = 1, \dots,m, 
\end{equation*}
has at most one positive real solution $x \in \RR^n_+$.
\item
$\sigma(\ker(A)) \cap \sigma(\im(B)) = \{0\}$.
\end{enumerate}
\end{proposition}

\begin{proof}
The left-hand side of the equation system in \refi is the image of $x$
under the generalized polynomial map $f_\kappa \colon \RR^n_+ \to \RR^m, x \mapsto A_\kappa \, x^B$.
Thus, statement \refi is equivalent to the injectivity of $f_\kappa$ for all $\kappa \in \RR^r_+$. 
So, by Corollary~\ref{cor:unrestricted}, \refi $\Leftrightarrow$ \refii.
\end{proof}

We can now prove the bound in the partial multivariate generalization of Descartes' rule.  

\begin{proof}[Proof of \refbound in Theorem~\ref{thm:Descartes}]
By Corollary~\ref{cor:det},
the hypothesis of \refbound in Theorem~\ref{thm:Descartes} is equivalent to statement \refii in Proposition~\ref{pro:polyequ} for $m=n$.
The equivalent condition \refi in Proposition~\ref{pro:polyequ} implies the conclusion of Theorem~\ref{thm:Descartes},
by setting $\kappa = (1, \ldots, 1)^T$.
\end{proof}

Next, we relate our results to the theory of oriented matroids.
With a vector configuration $A = (a^1, \ldots, a^r) \in \RR^{n \times r}$ of $r$ vectors spanning $\RR^n$,
one can associate the following data, each of which encodes the combinatorial structure of $A$.
On one side,
the {\em chirotope} of $A$, defined by the signs of maximal minors,
\begin{align*}
\chi_A \colon \{1, \ldots , r\}^n &\to \{-, 0, +\} \\
(i_1, \dots, i_n) &\mapsto \sign(\det(a^{i_1} , \ldots , a^{i_n})),
\end{align*}
records for each $n$-tuple of vectors whether it forms a positively oriented basis of $\RR^n$, 
forms a negatively oriented basis, or is not a basis.
On the other side, 
the set of {\em covectors} of $A$,
\[
\mathcal{V}^*(A) = \left\{ \left( \sign(t^T a^1 ), \ldots , \sign(t^T a^r) \right) \in \{-,0,+\}^r \st t \in \RR^n \right\},
\]
encodes the set of all ordered partitions of $A$ into three parts, induced by hyperplanes through the origin. 
Equivalently, the covectors of $A$ are the sign vectors of $A^T$,
\[
\mathcal{V}^*(A) = \sigma \left( \im(A^T) \right) ,
\]
since for $x = A^T t \in \RR^r$ with $t \in \RR^n$, we have
\[
\sigma(x)_i
= \sign(x_i)
= \sign\left( {\textstyle \sum_j} (A^T)_{ij} \, t_j \right)
= \sign\left( {\textstyle \sum_j} \, t_j a_{ji} \right)
= \sign( t^T a^i),
\] 
and hence $\sigma(x) = \left( \sign(t^T a^1 ), \ldots , \sign(t^T a^r) \right)$.

Further, the set of {\em vectors} of $A$, denoted by $\mathcal{V}(A)$, is the orthogonal complement of $\mathcal{V}^*(A)$.
We note that two sign vectors $\mu, \nu \in \{-, 0, +\}^n$ are {\em orthogonal}
if $\mu_i \nu_i = 0$ for all $i$ or if there exist $i,j$ with $\mu_{i} \nu_{i} = +$ and $\mu_{j} \nu_{j} = -$.
We have
\[
\mathcal{V}(A) = \mathcal{V^*}(A)^\bot = \sigma(\im(A^T))^\bot = \sigma(\im(A^T)^\bot) = \sigma(\ker(A)) ,
\]
where we use $\sigma(S)^\bot = \sigma(S^\bot)$ for any vector subspace $S \subseteq \RR^n$, cf.~\cite[Proposition~6.8]{Ziegler1995}.

The \emph{oriented matroid} of $A$ is a combinatorial structure
that can   
be given by any of these data (chirotopes, covectors, vectors)
and defined/characterized in terms of any of the corresponding axiom systems
\cite{BjornerLasSturmfelsWhiteZiegler1999, Richter_Ziegler, Ziegler1995}.
The proofs for the equivalences among these data/axiom systems are nontrivial.
We note that $\chi_A$ and $-\chi_A$ define the same oriented matroid.

We may now express the sign condition in Proposition~\ref{pro:polyequ} in terms of oriented matroids.
Clearly, $\sigma(\ker(A)) \cap \sigma(\im(B)) = \{0\}$
if and only if $\mathcal{V}(A) \cap \mathcal{V^*}(B^T) = \{0\}$.
In other words, no nonzero vector of $A$ is orthogonal to all vectors of $B^T$,
or, equivalently, no nonzero covector of $B^T$ is orthogonal to all covectors of $A$.

Analogously, we translate the sign conditions in statement~\refsurj of Theorem~\ref{thm:Descartes}.
Indeed, the maximal minors of $A$ and $B$ have the same (opposite) sign(s) if and only if $\chi_A = \pm \chi_{B^T}$,
that is, if and only if $A$ and $B^T$ define the same oriented matroid.

The proof of statement~\refsurj in Theorem~\ref{thm:Descartes} combines our injectivity result, Proposition~\ref{pro:polyequ}, 
with a surjectivity result from previous work~\cite[Theorem~3.8]{MR} to guarantee the existence and uniqueness of a positive solution.
In fact, \refsurj restates a generalization of Birch's theorem~\cite[Proposition~3.9]{MR} in terms of polynomial equations.

\begin{proof}[Proof of \refsurj in Theorem~\ref{thm:Descartes}]
Clearly, if \eqref{eq:y} has a positive solution $x \in {\RR^n_+}$, then $y \in C^\circ(A)$.
Conversely, the sign conditions in \refsurj,
together with the full rank of the matrices,
imply the hypotheses of~\refbound, and hence there is at most one positive solution of~\eqref{eq:y}.
Further, they imply that $A$ and $B^T$ define the same oriented matroid, and hence $\sigma(\im(A^T)) = \sigma(\im(B))$.
Finally, the assumption about the row vectors of $B$ implies the sign condition $(+,\ldots,+)^T \in \sigma(\im(A^T))$.

By~\cite[Theorem~3.8]{MR},
the generalized polynomial map $f_\kappa \colon \RR^n_+ \to C^\circ \subseteq \RR^n, x \mapsto A_\kappa \, x^B$
is surjective for all $\kappa \in \RR^r_+$.
Clearly, the left-hand side of the equation system~\eqref{eq:y} is the image of $x$
under the generalized polynomial map $f_\kappa$ for $\kappa = (1,\ldots,1)^T$.
Hence the equation system has at least one solution $x \in \RR_+^n$ for all $y \in C^\circ(A) \subseteq \RR^n$.
(We note that the relevant objects in~\cite[Theorem~3.8]{MR} are $F(\lambda) = f_\kappa(x)$
with $\lambda = \ln x$, $V=A^T$, $\tilde{V}=B$, and $c^*=\kappa$.)
\end{proof}

Observe that statement \refsurj in Theorem~\ref{thm:Descartes} can also be stated for a fixed exponent matrix $B \in \RR^{r \times n}$
with full rank $n$ and row vectors lying in an open half-space.
Then, for any coefficient matrix $A \in \RR^{n \times r}$ such that $A$ and $B^T$ define the same oriented matroid,
the equation system~\eqref{eq:y} has exactly one positive solution $x \in \RR_+^n$, for any $y \in C^\circ(A)$.
Alternatively, the hypotheses of statement \refsurj can be expressed in a more symmetric way: ``consider 
matrices $A \in \RR^{n \times r}$ and $B \in \RR^{r \times n}$ such that $A$ and $B^T$ define the same oriented matroid and the column vectors of $A$ (or, equivalently, the row vectors of $B$) lie in an open halfspace''.

\begin{remark} \label{rmk:homotopy}
In many applications, the existence of positive solutions is guaranteed,
for instance, as in item~\refsurj above or by a fixed-point argument,
in which case the sign condition in Proposition~\ref{pro:polyequ} suffices to ensure the existence and uniqueness of positive solutions.
In this setting, homotopy continuation methods can be used to obtain the solution for a given system~\cite{SommeseWampler}.  
Namely, we identify one system in the family that has a unique solution
-- by choosing the coefficients $a_{ij}$ and right-hand sides $y_i$ appropriately, we can ensure that $x=(1,\dots,1)$ is the unique solution --
and then we follow the unique positive solution while performing the homotopy by deforming the parameters of the solved system to those of our given system.  
However, this need not always work, since the followed solution can fail to remain positive along the way.
\end{remark}


\section{Algorithmic verification of sign conditions} \label{sec:algo}

\newcommand{\QQ}{{\mathbb Q}}
\newcommand{\T}{{T}}

In this section we outline how the sign condition \refsig in Theorem~\ref{thm:main} can be verified algorithmically.
Recall that for matrices $A\in \RR^{m \times r}$, $B\in \RR^{r \times n}$, and a subset $S \subseteq \RR^n$,
the condition
\[
\text{\refsig} \qquad \sigma(\ker(A)) \cap \sigma(B(\Sigma(S^*))) = \emptyset
\]
is equivalent to the injectivity of $f_\kappa(x) = A_\kappa \, x^B$ with respect to $S$, for all $\kappa \in \RR^r_+$.
A characterization of condition \refsig  in terms of determinants and signs of maximal minors  is given in 
Theorem~\ref{thm:jacdet} for the special case where $S$ is a vector subspace with $\dim (S)=\rank(A)$.

We assume that the two matrices have rational entries:
$A\in \QQ^{m \times r}$ and $B\in \QQ^{r \times n}$. As discussed in the introduction, $f_\kappa$ 
is injective with respect to $S$ if and only if  it is injective with respect to any subset $S'$ 
for which $S\subseteq S'\subseteq \Sigma(S) = \sigma^{-1}(\sigma(S))$. The subset $\Sigma(S)$ 
depends only on the set of nonzero sign vectors $\sigma(S)$ of $S$.
Therefore, we assume that a set of sign vectors $T \subseteq \{-,0,+\}^n\setminus \{0\}$ 
is given and discuss how to  check condition \refsig 
for the corresponding union of (possibly lower-dimensional) orthants $\sigma^{-1}(T)$, that is, whether 
\[
\text{\refsig} \qquad \sigma(\ker(A)) \cap \sigma(B(\sigma^{-1}(T))) = \emptyset
\]
 holds.
Clearly, \refsig holds if and only if there do not exist sign vectors $\mu \in \{-,0,+\}^r$ and $\tau \in \T$ such that 
\[
\mu \in \sigma(\ker(A)) \cap \sigma(B(\sigma^{-1}(\tau))),
\]
or, equivalently, if for all $\mu \in \{-,0,+\}^r$ and all $\tau \in \T$ the system of linear inequalities 
\begin{equation}
\label{eq:system}
Ax=0, \quad \sigma(x)=\sigma(By)=\mu, \quad \sigma(y)=\tau
\end{equation}
is infeasible, that is, the system~\eqref{eq:system} has no solution $z=(x,y) \in  \RR^{r+n}$. 

Linear inequalities arise from the sign equalities  in \eqref{eq:system}. 
Some of these are strict inequalities and hence techniques from linear programming do not directly apply. 
However, since the inequalities are homogeneous, the set of solutions to \eqref{eq:system} forms a convex cone. 
In particular, if $z$ is a solution, then so is $ \lambda  z$, for all $\lambda \in \RR_+$.  Therefore, 
we can verify the infeasibility of \eqref{eq:system} by checking the infeasibility of the system of linear 
inequalities obtained by replacing the inequalities $>0$ and $<0$ by $\geq \epsilon$ and $\leq -\epsilon$, 
respectively, for an arbitrary $\epsilon \in \RR_+$. In this setting, one can apply methods for exact linear 
programming, which makes use of Farkas' lemma
to guarantee the infeasibility of linear programs by way of rational certificates; 
see for  example~\cite{ApplegateCookDashEspinoza2007,AlthausDumitriu2012,GleixnerSteffyWolter2012} 
and the exact linear programming solver \textsc{QSopt\_ex}~\cite{ApplegateCookDashEspinoza2009}. 
An alternative is to develop and adapt exact linear programming methods for strict inequalities using  
Theorems of the Alternative (Transposition theorems); see for example~\cite{Mangasarian1994,Schrijver1986}.
Using this approach, we might need to test the infeasibility of  system \eqref{eq:system} for $3^r$ 
times the cardinality of $T$ choices of pairs $\mu \in \{-,0,+\}^r$ and $\tau \in \T$.

To apply this approach, we need to compute the set of sign vectors $\sigma(S)$ of $S$. 
In the applications in Section~\ref{sec:app}, the subset $S$  is a vector subspace. 
In this case, the set of sign vectors $\sigma(S)$ are the covectors of the corresponding 
oriented matroid. Chirotopes can be used to compute covectors with minimal support, 
which are called cocircuits. Covectors can be computed from cocircuits. 
In general, the number of covectors can be exponentially large compared to the number of cocircuits.
For example, $\RR^n$ has $3^n$ covectors and $n$ cocircuits corresponding to the vectors of the standard basis.
Therefore it is reasonable to measure the complexity of enumeration algorithms as a function of input and output sizes. 
By this measure, an efficient polynomial-time algorithm that generates all covectors 
from cocircuits is discussed in~\cite{BabsonFinschiFukuda2001}.  Note that one also can use chirotopes to test 
directly whether the oriented matroids corresponding to two vector subspaces are equal, which is the condition for 
existence and uniqueness of positive real solutions in item~\refsurj of Theorem~\ref{thm:Descartes}.

For the special case of unrestricted injectivity (cf.\ Corollary~\ref{cor:unrestricted} and Proposition~\ref{pro:dual}),  
condition~\refsig reduces to the condition 
\begin{equation*}
\label{eq:intersect}
\sigma(\ker(A)) \cap \sigma(\im(B)) = \{0\},
\end{equation*}
for matrices $A \in \RR^{m \times r}$ and $B \in \RR^{r \times n}$, such that $A$ has full rank $m$ and $B$ has full rank $n$. 
In other words, we must check whether the two vector subspaces $\ker(A) $ and $\im(B)$ 
have a common nontrivial sign vector, or, equivalently, whether the corresponding oriented matroids have a common covector. 
For $m=n$, this condition is characterized in Corollary~\ref{cor:det} in terms of signs of products of maximal minors. 
For $m>n$, it is shown in~\cite[Theorem~5.5]{Chaiken1996} that for integer matrices the problem is strongly NP-complete. 

We cannot hope for a polynomial-time algorithm to verify condition \refsig in general.
A software to find nonzero sign vectors in  $\sigma(\ker(A)) \cap \sigma(\im(B))$  is described in~\cite{Uhr2012}. 
It uses mixed linear integer programming and branch-and-bound methods for enumerating all sign vectors,  
and has been successfully  applied to establish multistationarity for models arising in Systems Biology~\cite{fein-043}.
The \texttt{C++} package \textsc{Topcom} \cite{Rambau2002} efficiently computes chirotopes with rational arithmetic and 
generates all cocircuits. It also has an interface to the open source computer algebra system \textsc{Sage}~\cite{sage}. 
For algorithmic methods to compute sign vectors of real algebraic varieties and semialgebraic sets, we refer to~\cite{BasuPollackRoy2006}. 
The software package \textsc{RAGLib} \cite{SafeyElDin2013} can test  
whether a system of polynomial equations and inequalities has a real solution.

\paragraph{Acknowledgments.}

This project began during the Dagstuhl Seminar on
``Symbolic methods for chemical reaction networks''
held in November 2012 at Schloss Dagstuhl, Germany.  
The authors also benefited from discussions during the AIM workshop on ``Mathematical problems arising from biochemical reaction networks'' 
held in March 2013, in Palo Alto.  
EF was supported by a postdoctoral grant ``Beatriu de Pin\'os'' from the Generalitat 
de Catalunya and  the project  MTM2012-38122-C03-01/FEDER from the Ministerio de Econom\'{\i}a y Competitividad (Spain).
CC was supported by BMBF grant Virtual Liver (FKZ 0315744)
and the research focus dynamical systems of the state
Saxony-Anhalt.
AS was supported by the NSF (DMS-1004380 and DMS-1312473). 
AD was partially supported by UBACYT 20020130100207BA, CONICET PIP 11220110100580 and ANPCyT PICT-2013-1110, Argentina. 
The authors thank an anonymous referee for helpful comments.

\end{document}